\documentclass[a4paper,reqno]{amsart}
\pdfoutput=1
\usepackage[T1]{fontenc}
\usepackage{lmodern}

\usepackage{amsmath,amsfonts,amssymb,amsthm,stmaryrd,bbm,graphicx,mathtools,enumerate}
\usepackage{mathrsfs}
\usepackage[T1]{fontenc} 
\usepackage[latin1]{inputenc}
\usepackage[english]{babel}
\usepackage[tracking,spacing,kerning,babel]{microtype}
\usepackage{wasysym} 
\usepackage{color}
\usepackage{xcolor}
\usepackage{framed} 
\usepackage{xparse}
\usepackage{mathtools}
\usepackage{relsize}
\usepackage{soul}

\usepackage{wrapfig} 

\usepackage[final]{hyperref}   
\hypersetup{
    linktoc=page,
    linkcolor=red,          
    citecolor=blue,        
    filecolor=blue,      
    urlcolor=cyan,
   colorlinks=true           
}

\usepackage{lipsum} 

\usepackage[toc,page]{appendix}

\usepackage{minitoc}
\usepackage{multicol}

\newtheoremstyle{mine}
{\baselineskip}
{\baselineskip}
{\itshape}
{
}
{\bfseries}
{.}
{.5em}
{#1 #2\ifx#3\relax\else~(#3)\fi}

\theoremstyle{mine}

\newtheorem{theorem}{Theorem}
\numberwithin{theorem}{section}
\newtheorem{corollary}[theorem]{Corollary}
\newtheorem{proposition}[theorem]{Proposition}
\newtheorem{lemma}[theorem]{Lemma}
\newtheorem{claim}[theorem]{Claim}
\newtheorem{definition}[theorem]{Definition}

\numberwithin{equation}{section}

\theoremstyle{remark}
\newtheorem{remark}{Remark}

\colorlet{shadecolor}{blue!10}

\let\qed=\QED
\let \ln=\log

\renewcommand{\epsilon}{\varepsilon}

\newcommand{\F}{\mathcal{F}}
\newcommand{\R}{\mathbb{R}}
\newcommand{\D}{\mathbb{D}}

\newcommand{\Z}{\mathbb{Z}}
\newcommand{\N}{\mathbb{N}}
\renewcommand{\S}{\mathbb{S}}

\newcommand{\1}{\mathbf 1}
\renewcommand{\d}{\mathbf d}

\renewcommand{\L}{\mathcal L}

\newcommand{\K}{\mathbf K}
\NewDocumentCommand{\insquare}{omo}{%
	\begingroup
	\IfValueTF{#1}{%
		\setlength{\fboxsep}{#1}%
	}{%
	}%
	\IfValueTF{#3}{%
		\setlength{\fboxrule}{#3}%
	}{}%
	\ensuremath{\fbox{#2}}
	\endgroup 
}
\newcommand{\An}{\insquare[0.9pt]{$\mathsmaller{\square}$}}

\def\T{\mathbb{T}}

\def\calK{\mathcal{K}}

\def\SLE{\mathrm{SLE}}
\def\CLE{\mathrm{CLE}}
\def\diam{\mathrm{diam}}

\def\P{\mathbb{P}} 
\newcommand{\E}{\mathbb{E}} 
\def\md{\mid}

\def \eps {\epsilon}

\def\Bb#1#2{{\def\md{\bigm| }#1\bigl[#2\bigr]}}

\def\Pb{\Bb\P}
\def\Eb{\Bb\E}

\def\FK#1#2#3{{\def\md{\bigm| } \P_{#1}^{\,#2}  \bigl[  #3 \bigr]}}
\def\EFK#1#2#3{{\def\md{\bigm| } \E_{#1}^{\,#2}  \bigl[  #3 \bigr]}}

\def \p {{\partial}}

\def\<#1{\langle #1\rangle}

\definecolor{darkgreen}{rgb}{0,0.6,0.05}

\def\nn{\nonumber}
\def\bi{\begin{itemize}}  
\def\ei{\end{itemize}}
\def\bnum{\begin{enumerate}} 
\def\enum{\end{enumerate}}
\def\ni{\noindent}

\title[2D Heisenberg model and exit sets of vector valued GFF]
{Percolation for 2D classical Heisenberg model and exit sets of vector valued GFF}

\author{Juhan Aru}
\author{Christophe Garban}
\author{Avelio Sepúlveda}

\address
{École Polytechnique Fédérale de Lausanne (EPFL), Institute of Mathematics, CH-1015 Lausanne, Switzerland}
\email{juhan.aru@epfl.ch}

\address
{Université Claude Bernard Lyon 1, CNRS UMR 5208, Institut Camille Jordan, 69622 Villeurbanne, France \, and Institut Universitaire de France (IUF)}
\email{garban@math.univ-lyon1.fr}

\address{Universidad de Chile,  Centro de Modelamiento Matemático (AFB170001), UMI-CNRS 2807, Beauchef 851, Santiago, Chile.}
\email{lsepulveda@dim.uchile.cl}

\usepackage{pgffor}
\begin{document}
	\maketitle
	\begin{abstract}

Our motivation in this paper is twofold. First, we study the geometry of a class of exploration sets, called {\em exit sets}, which are naturally associated with a 2D vector-valued Gaussian Free Field : $\phi : \Z^2 \to \R^N, N\geq 1$. We prove that, somewhat surprisingly, these sets are a.s. degenerate as long as $N\geq 2$, while they are conjectured to be macroscopic and fractal when $N=1$. 

\smallskip 
This analysis allows us, when $N\geq 2$, to understand the percolation properties of the level sets of $\{ \|\phi(x)\|, x\in \Z^2\}$ and leads us to our second main motivation in this work: if one projects a spin $O(N+1)$ model (the case $N=2$ corresponds to the classical Heisenberg model) down to a spin $O(N)$ model, we end up with a spin $O(N)$ in a quenched disorder given by random conductances on $\Z^2$. Using the {\em exit sets} of the $N$-vector-valued GFF,  we obtain a local and geometric description of this random disorder in the limit $\beta\to \infty$.  This allows us in particular to revisit a series of celebrated works by Patrascioiu and Seiler (\cite{patrascioiu1992phase,patrascioiu1993percolation,patrascioiu2002percolation}) which argued against Polyakov's prediction that spin $O(N+1)$ model is massive at all temperatures as long as $N\geq 2$ (\cite{polyakov1975interaction}). We make part of their arguments rigorous and more importantly we provide the following counter-example: we build ergodic environments of (arbitrary) high conductances with (arbitrary) small and disconnected regions of low conductances in which, despite the predominance of high conductances, the $XY$ model remains massive. 

\smallskip 
Of independent interest, we prove that at high $\beta$, the fluctuations of a classical Heisenberg model near a north pointing spin are given by a $N=2$ vectorial GFF. This is implicit for example in \cite{polyakov1975interaction} but we give here the first (non-trivial)  rigorous proof. Also, independently of the recent work \cite{dubedat2022random}, we show that two-point correlation functions of the spin $O(N)$ model can be given in terms of certain percolation events in the {\em cable graph} for any $N\geq 1$. 
	\end{abstract}

\begin{figure*}
	\includegraphics[width=0.7\textwidth]{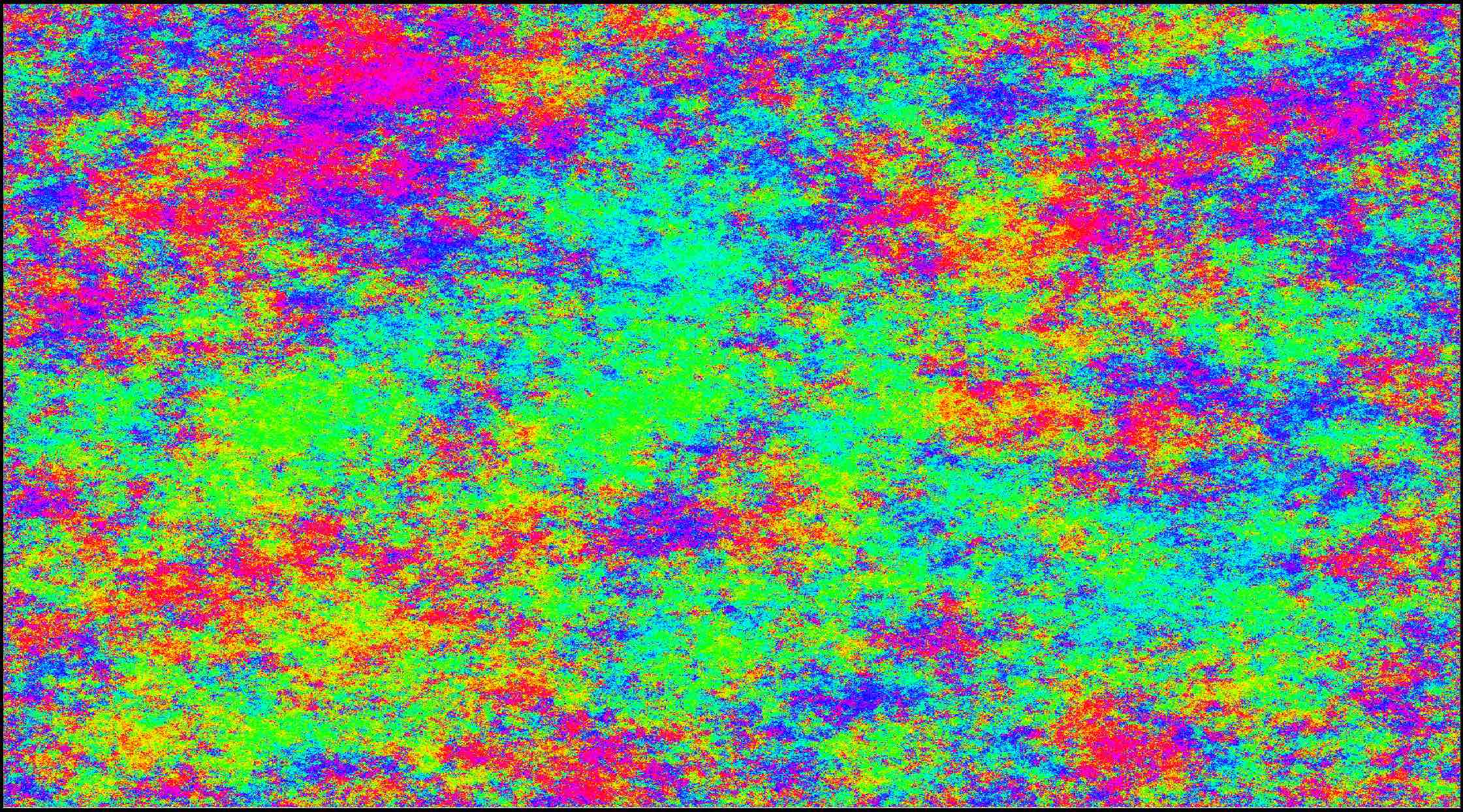}
	\caption{Angles of a vectorial GFF in $[0,15000]\times [0,15000]\subset \Z^2$. They are the protagonists of this paper.}
\end{figure*}

\section{Introduction}

\subsection*{Classical Heisenberg model.}  The spin $O(N)$ model on $\Z^d$ is a fundamental model in statistical physics which includes many celebrated models such as the Ising model ($N=1$), the plane rotator or $XY$ model ($N=2$) and the classical Heisenberg model ($N=3$). On a finite box $\Lambda\subset \Z^d$, its state space is given by $(\S^{N-1})^\Lambda$ (where $\S^{N-1}$ is the unit sphere in $\R^N$) and the interaction between spins is described for any $\sigma\in (\S^{N-1})^\Lambda$ by the following Hamiltonian (in the case of free boundary conditions)
\begin{align*}\label{}
H_\Lambda(\sigma):= - \sum_{i\sim j} \sigma(i)\cdot \sigma(j)\,,
\end{align*}
where the sums runs over neighbouring sites in $\Lambda$ and where $\cdot$ is the scalar product in $\R^N$. As always in statistical physics, the inverse temperature $\beta$ is parametrizing the family of Boltzmann measures 
\begin{align*}\label{}
\mu_{\Lambda, \beta}(d \sigma) \propto  \exp\left (\beta \sum_{i\sim j} \sigma(i)\cdot \sigma(j)\right ) \prod_{i\in \Lambda} \lambda_{\S^{N-1}}(d\sigma(i))\,,
\end{align*}
where $\lambda_{\S^{N-1}}$ stands for the uniform measure on the sphere $\S^{N-1}$. 
We refer the reader for example to \cite{friedli2017statistical,peled2019lectures} for excellent references on the spin $O(N)$ model. Its rich behaviour depends both on the dimension $d$ of the lattice as well as the parameter $N\in \N^+$ which indicates the $O(N)$ symmetry of the model. We shall focus in this paper only on the two-dimensional case $d=2$ (see the above references for a description of the rich phenomenology when $d\geq 3$ as well as \cite{FSS1976,garban2022continuous}). In the planar case, we distinguish three possible behaviours depending on the value of $N$ (the third scenario being only conjectural).
\bi
\item $N=1$ (Ising model). In this case it is well known since Peierls (\cite{peierls1933theorie}) that there is long-range order at low temperature. 
See for example the recent survey \cite{duminil2022100} for a broad overview of the recent intense research activity about the Ising model.

\item $N=2$ ($XY$ model). Due to its continuous symmetry, Mermin and Wagner proved in the 60's (\cite{mermin1966absence,mermin1967absence}) that such a spin system does not have long-range order at any $\beta >0$. (This also holds for any $N\geq 2$). Berezinskii, Kosterlitz and Thouless predicted in the 70's that this model should nevertheless undergo a phase transition with {\em quasi long-range order} at low temperatures. This is now known as the BKT transition and its existence was rigorously proved by Fröhlich and Spencer in the seminal paper \cite{FS}. See also the very interesting new proofs of this phase transition in \cite{van2021elementary,aizenman2021depinning}. 
\item $N\geq 3$ (including the classical Heisenberg model, $N=3$). In this case, it has been predicted since \cite{polyakov1975interaction} that this model should exhibit exponential decay of correlations at all inverse temperatures $\beta>0$. This remains unproven and it is considered to be one of the major conjectures in statistical physics (\cite{simon1984fifteen}). Even though the existence of a mass at arbitrary low temperatures is widely accepted across the theoretical physics community, some physicists have discussed the validity of this prediction, most prominently the series of works by Patrascioiu and Seiler 
\cite{patrascioiu1992phase,patrascioiu1993percolation,patrascioiu2002percolation}. Their arguments are very legitimate objections to what may go wrong in the RG approximation of Polyakov's argument \cite{polyakov1975interaction}. One major motivation of this work is to revisit their analysis using recent tools developed for level sets of the Gaussian free field. Let us also stress that numerical simulations are not very successful so far to help  deciding between both possible scenarios (power law decay  versus massive scenario) as they are facing the rapid divergence of the correlation length as $\beta \to \infty$. 
\ei

Patrascioiu-Seiler's approach roughly goes as follows (see Section \ref{s.PS} for a more detailed description of their works \cite{patrascioiu1992phase,patrascioiu1993percolation,patrascioiu2002percolation}). Similarly as  in \cite{polyakov1975interaction}, the idea is to project the spin system $(\sigma(i))_{i\in \Z^2} = (\theta^1(i),\theta^2(i),\theta^3(i))_{i\in \Z^2}$ down to a $XY$ model in $(\S^1)^{\Z^2}$ by conditioning on the value of the third coordinates $\{\theta^3(i)\}_{i \in \Z^2}$. It easy to see that conditioned on this partial information, the $\S^1$-valued spins $\{ \frac 1 {\sqrt{1 - \theta^3(i)^2}} (\theta^1(i), \theta^2(i))\}_{i\in \Z^2}$ are distributed as an $XY$ model in random conductances given by 
\begin{align}\label{e.Cij}
C_{ij}  =  \beta \sqrt{1-\theta^3(i)^2} \sqrt{1-\theta^3(j)^2}\,. 
\end{align}
(See Proposition \ref{pr.projection} for a more precise statement).  
By slightly modifying the interaction between neighbouring spins (and thus slightly modifying the model) in a way which still preserves the $O(3)$ symmetry, their analysis is essentially divided as follows:
\begin{enumerate} [A)]
\item Either, the set of edges which carry a very large conductance (say bigger than $\sqrt{\beta}$) is not percolating. In this case, modulo some natural hypothesis, they show 
that exponential decay is impossible. 
\item If on the other hand, this set of large conductances is percolating, then for $\beta$ large one should enter a BKT phase with power-law decay. 
\end{enumerate}
Most of the works \cite{patrascioiu1992phase,patrascioiu1993percolation,patrascioiu2002percolation} focus on scenario A) as their authors believe that scenario B) should not happen. In this present work, with a slightly different setup (namely we work on the true Heisenberg model, but our cut-offs for scenarios A) and B) are different, as well as our definition of edges with large conductances) our main contributions about Patrascioiu-Seiler's analysis may be summarized as follows:
\bi
\item[\~ A)] We give a geometric interpretation of spin to spin correlation functions in Theorem \ref{th.corr}. Then in Corollary \ref{th.cable}, we prove a rigorous version of their scenario A) with a slightly different notion of ``edge carrying a large conductance''. 
\item[\~ B)] In Theorem \ref{th.example} below, for any fixed $\beta$, we build ergodic environments of random conductances such that the set of conductances $\{C_{ij}, C_{ij}\geq \beta\}$ is a.s. percolating in a strong sense and yet, the $XY$ model in this field of conductances is massive.  (Note that \~B does not cover all cases where \~A does not apply). 
\ei

It turns out that these two results which revisit Patrascioiu-Seiler's program rely both on tools from the analysis  of vector-valued Gaussian free field as well as Brownian loop soups. The work \cite{polyakov1975interaction} in fact already relied on such a link between 2-component GFF and spin $O(3)$ model, and we make this link rigorous in Theorem \ref{th.PolyakovEasy} below. Consequentely, we now turn to the second main motivation of this paper which deals with the geometric analysis of two-dimensional vector-valued Gaussian free fields. 

\subsection*{Local sets of scalar and vector-valued Gaussian free field.}
Fix $N\geq 1, \beta > 0$, a finite domain $\Lambda \subset \Z^2$ and a boundary set $\emptyset \neq \p \Lambda \subset \Lambda$. The $N$-vector valued Gaussian free field (GFF) on $\Lambda$ at inverse temperature $\beta$ with 0-boundary conditions on $\p \Lambda$ is the field $\phi : \Lambda \to \R^N$ whose density is proportional to 
\begin{align*}\label{}
\exp\left (-\frac \beta 2 \sum_{i\sim j} \| \phi(i) - \phi(j)\|^2\right )\,.
\end{align*}
(See Definition \ref{d.GFF}). 
In the scalar case ($N=1$), the Gaussian free field has played a central role over the last twenty years in statistical physics. See for example \cite{sheffield2005random,schramm2009contour}. The analysis of its {\em level lines} has proved being very rich both for the GFF defined on a lattice  as in the present work (\cite{schramm2009contour}) or for its continuum analog defined, say on $[0,1]^2$ (\cite{schramm2013contour,WaWu,PW}). In both cases, level lines of the GFF are described by versions of $\SLE_4$ and $\CLE_4$. From then on, a systematic study of all possible {\em exploration sets} of the scalar Gaussian free field (in its discrete and continuum version) has been initiated. In the continuum, where this concept is more subtle, the right notion has been axiomatized under the name of {\em local sets} in \cite{schramm2013contour,dubedat2009sle,MS1} and was further investigated in \cite{aru2019bounded}. 

The  local sets (in the scalar case $N=1$) that are more relevant to this paper are the so-called {\em two-valued sets} which were introduced in \cite{aru2019bounded} and subsequently studied in \cite{AS2,ALS1,ALS2, schoug2022dimensions}. In the discrete setting, they may be defined for any $a,b>0$ informally as follows: we explore the values of the field $\phi$ starting from $\p \Lambda$ (on which $\phi \equiv 0$) and we keep exploring ``inward'' as long as $-a\leq \phi(x) \leq b$. Each time the exploration process visits a vertex $x$ at which $\phi(x)\in[a,b]^c$, the exploration process stops at that vertex. Defined this way, one obtains a random subset of $\Lambda$ which is called $\mathbb{A}_{-a,b}$. (For a more formal definition in the $N$-component case see \eqref{e.ARk} and Section 6 of \cite{aru2019bounded} for the continuum GFF). It turns out, perhaps surprisingly, that the geometric understanding of these sets is still lacking in the 2-dimensional discrete setting (see \cite{PFSPA} for a study of its percolation properties in $\Z^d$ for $d \geq 3$), while in the continuum setting, much more is known. For example, they are known to exist and be unique iff $a+b \geq \sqrt{\pi/2}$, their boundaries are locally SLE$_4$ type \cite{ASW}, the intersection properties of their different boundary components can be precisely described \cite{AS2} and  their almost sure Hausdorff dimension is computed in \cite{schoug2022dimensions} to be $1-\frac{\pi}{4(a+b)^2}$.  We also point out that in the discrete setting, {\em one-valued sets} are known to be macroscopic by \cite{LupuWerner2016Levy}, see also \cite{ding2022crossing} for an alternative argument.

\smallskip
We now turn to $N$-vector valued GFF on a graph $\Lambda \subset \Z^2$ which opens the way to new classes of local sets when $N\geq 2$.  In the discrete setting, we will focus in this paper on the {\em exit sets} which may be defined informally as follows for any radius $R>0$: we explore the field $\phi$ inward starting from $\p \Lambda$ and we keep exploring as long as $\|\phi(x)\| \leq R$. The exploration stops at each vertex where $\| \phi(x) \|>R$. We obtain this way a random subset of $\Lambda$ which we shall call $A_R=A_R(\phi)$. More precisely, 
\begin{align}\label{e.AR}
A_R:=
\Big\{  x \in \Lambda,  &\exists \, \gamma=\{x_0, \ldots, x_m\}, \text{ s.t. } x_0\in \p \Lambda, x_i\sim x_{i+1} \forall 0\leq i \leq m-1,  \nn \\  &  x_m\sim x \text{ and } \|\phi(x_i)\|\leq R, \,\, \forall 1 \leq i \leq m \Big\}
\end{align}
For any integer $k\geq 1$, we will also consider the following extension which allows ``jumps'' of distance $\leq k$ along the exploration. Namely, we define 
\begin{align}\label{e.ARk}
A_{R,k}:=
\Big\{  x \in \Lambda,  &\exists \, \gamma_k=\{x_0, \ldots, x_m\}, \text{ s.t. } x_0\in \p \Lambda, d(x_i,x_{i+1}) \leq k, \forall 0\leq i \leq m-1,  \nn \\  &  d(x_m,x)\leq k  \text{ and } \|\phi(x_i)\|\leq R, \,\, \forall 1 \leq i \leq m \Big\}
\end{align}
Note that with these notations, we have $A_{R}=A_{R,k=1}$. 

Another main motivation in this work, which is related to the above context of classical Heisenberg model but which is also interesting on its own, is the geometric study of such exploration sets as well as its consequences on the level sets of $N$-vectorial GFF and massive $N$-vectorial GFF in the plane.

\subsection*{Main results.} 

Our first main result is about the degeneracy of the above {\em exit sets} $A_{R,k}=A_{R,k}(\phi)$ for any $N\geq 2$. See Figure \ref{f.main_idea} for a summary of the main idea behind this result.
\begin{theorem}\label{th.ExitSet}
Let $N\geq 2$. For any $\eps>0$ and any $R>0, k\geq 1$, if $\phi : \Lambda_n=[-n,n]^2 \to \R^N$ is a vectorial GFF with zero-boundary conditions on $\Lambda_n$ and if $A_{R,k}$ is the exit set defined in~\eqref{e.ARk}, then 
\begin{align*}\label{}
\lim_{n \to \infty}  \Pb{A_{R,k} \cap \Lambda_{(1-\eps) n}} = 0\,. 
\end{align*}
More quantitatively, we may allow the thresholds $R$ and $k$ to depend on the scale $n$:
\begin{align*}\label{}
\lim_{n \to \infty}  \Pb{A_{R(n),k(n)} \cap \Lambda_{(1-\eps) n}} = 0\,,
\end{align*}
as long as $(R(n) + \log k(n))^6 = o(\log n)$.
\end{theorem}
\begin{figure}
	\includegraphics[width=0.48\textwidth]{Angle2GFF.jpg}
	\includegraphics[width=0.48\textwidth]{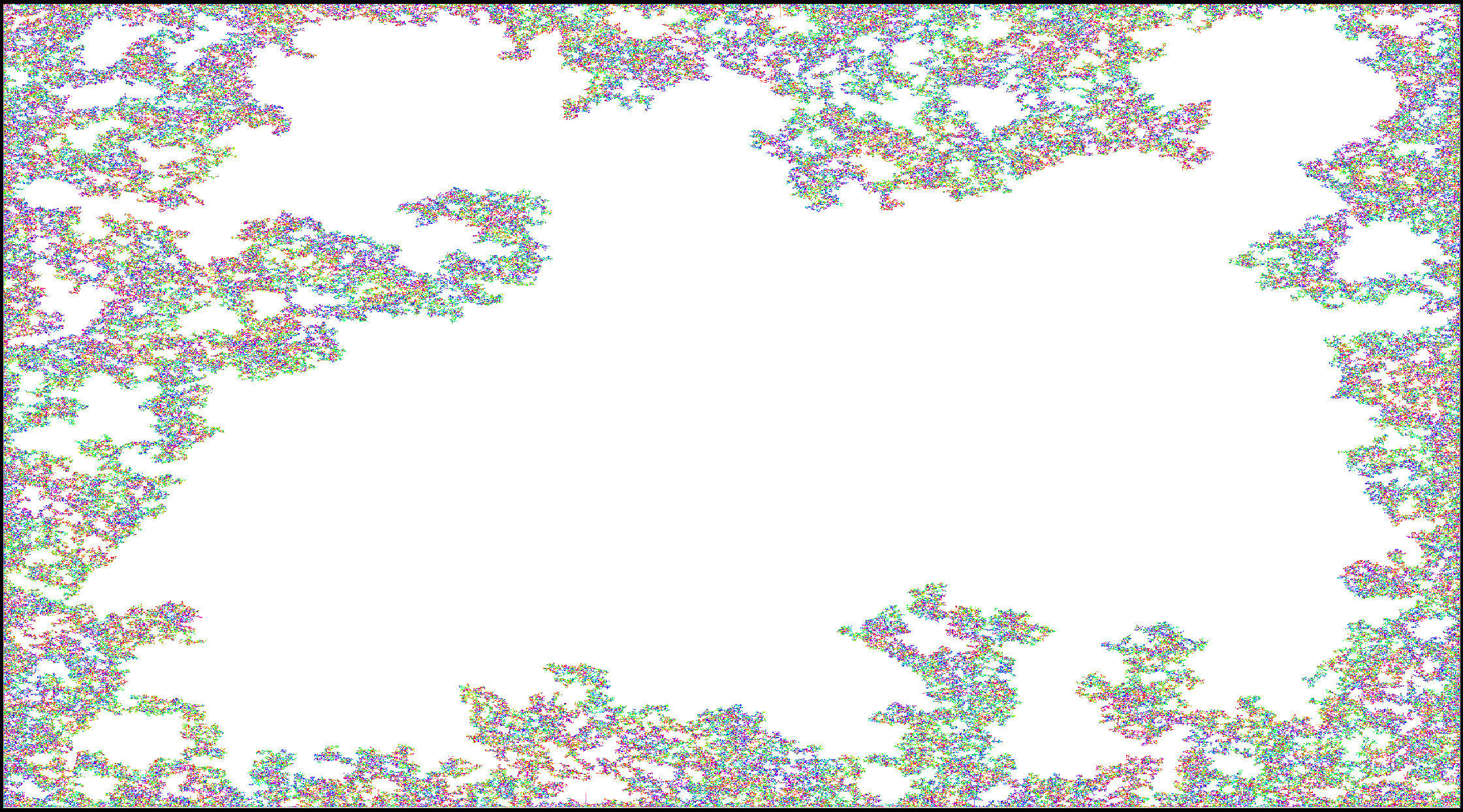}
	\caption{These simulations represent the main idea used to prove Theorem \ref{th.ExitSet}. To the left, we see the angles of a 2D vectorial GFF (15000$\times$ 15000), to the right we see  $A_{1,k}$ with $k$ being the smallest one that is ``macroscopic''. We see that even though the angles of the GFF have big connected components of almost constant angle (colour in the image), the angles in the exploration are fairly evenly distributed. This fact allows us to use a ``law of large numbers'' to  show that the harmonic extension of the values in the boundary are much smaller than the trivial bound.}
	\label{f.main_idea}
\end{figure}

This theorem shows that as opposed to the case $N=1$, one should not expect non-trivial limiting local sets for a continuum vectorial GFF. 

Our second main result describes the percolation properties of the sets $\{ \|\phi(x)\| < R\}_{x\in \Z^2}$ for both a massive and a non-massive GFF pinned at the origin. In this light, denote by $\{x \overset{\| \cdot \| \leq R}\longleftrightarrow  y \}$ the event that one can find a path connecting $x$ and $y$ on which the norm of $\phi$ stays below $R$ all along, and denote by $\{x \overset{\| \cdot \| \leq R, k}\longleftrightarrow  y \}$ the same event when we allow jumps up to distance $k$ on the path. 

\begin{theorem}\label{th.GFFmGFF}
Let $N\geq 2$. 
\bi
\item Consider the $N$-vectorial GFF on $\Z^2$ which is rooted at the origin. 
Then, for any $R > 0, k \geq 1$, there exists $\psi=\psi(R,k)>0$, such that for any $x,y\in \Z^2$, we have the following exponential decay:
\begin{align*}\label{}
\Pb{x \overset{\| \cdot \| \leq R,k}\longleftrightarrow  y } \leq e^{-\psi(R,k) \|x-y\|_2}\,.
\end{align*}
\smallskip

\item For a massive $N$-vectorial GFF $\phi^{(m)}$ on $\Z^2$, we have that for any $R>0, k\geq 1$, there exists a sufficiently small mass $\bar m=\bar m(R,k)>0$ and a positive $\psi(R,k)>0$ such that for any $m\leq \bar m$ and any  $x,y\in \Z^2$,
\begin{align*}\label{}
\FK{\phi^{(m)}}{}{x \overset{\| \cdot \| \leq R,k}\longleftrightarrow  y  \text{ in } \Z^2} \leq e^{-\psi(R,k) \|x-y\|_2}\,.
\end{align*}
\ei
\end{theorem}

\begin{remark}\label{}
This theorem can easily be shown to be true when $R$ (or $R=R(m)$) is sufficiently small even when $N=1$. On the other hand, for $R$ large, it is believed to be wrong in the scalar case $N=1$. 
\end{remark}

\begin{remark}\label{}
As detailed in Sections \ref{s.AR} and \ref{s.PS},  the reason why we need to relax our definition of connectedness in order to allow jumps of size $k\in \N^+$ is twofold:
\bnum
\item First, because of our percolation considerations. Indeed, since site-percolation is not self-dual and since the value of conductances $C_{ij}$ defined in~\eqref{e.Cij} involve two vertices, these two cases require to consider $k=2$ when we analyse their clusters. 
\item Second, because in the proof of Theorem \ref{th.example} below, we will need to show that the Ising model in the sub-graph of high conductances has long-range order for non-diverging inverse temperatures ($\beta=1$). To achieve this, will need to rely on diverging jumps $k(n)=C \log n$ while exploring the GFF exit sets. 
\enum
\end{remark}

Our third main result below makes rigorous for the first time the easier part in Polyakov's argument \cite{polyakov1975interaction}. This is definitively not the interesting part of \cite{polyakov1975interaction} and the statement below was considered {\em folklore} in the physics community. Nevertheless, we realised its proof is non-trivial and requires to exclude possible additional fluctuations coming from random harmonic functions in the whole punctured plane $\Z^2\setminus \{0\}$. See Section \ref{s.Taylor}. 

\begin{theorem}\label{th.PolyakovEasy}
	Let $\theta = (\theta^1, \ldots, \theta^N)$ denotes a spin $O(N)$ model on the torus $\T_{n}$ at inverse temperature $\beta$  rooted to point north at $0$, (i.e. such that $\theta^N(0) = 1$).
	
	For any sequence of $n=n_\beta$ with $n_\beta \to \infty$ as $\beta \to \infty$, the rescaled vector $\sqrt{\beta}(\theta^1, \ldots, \theta^{N-1})$ converges in law as $\beta\to \infty$  to the $(N-1)$-vectorial GFF rooted to be $0$ at $0$. Here, the topology corresponds to convergence in law  on compacts subsets of $\T_{n_\beta}$.
\end{theorem}

Combining this result with Theorem \ref{th.GFFmGFF} we obtain the following corollary. 
\begin{corollary}\label{c.drama}
For any $N\geq 3$ and $\beta>0$, let us consider any infinite volume limit of the spin $O(N)$ model on the torus $\T_n^2$ (which is then a translation invariant measure) and let us globally rotate spins so that $\theta^N(0)=1$ (i.e. a north pointing spin at the origin).  

Following the procedure described above, if we project the system down to an $XY$ model in random conductances $C_{ij}$ given in~\eqref{e.Cij}. Then for any $R>0, \, k\in \N^+$, as $\beta\to \infty$, the regions in the vicinity of the origin of high local conductance, i.e the set of edges
\begin{align*}\label{}
\{ e=(ij), \text{ with } C_{ij} > R \} 
\end{align*}
is a.s. strongly percolating (in the sense that its complement 
is exponentially clustering).  
\end{corollary}

This result is of course very far from proving Polyakov's prediction as it only describes the law of the random environment in a $\beta$-dependent neighbourhood of $0$ 
(see Remark \ref{r.quant}). Still it shows that around a typical north-pointing spin, the spins behave in such a way that very cold edges (conductances bigger than a given  arbitrary large threshold $R$) are strongly percolating in the vicinity of that north-pointing edge.

In order to reconcile Corollary \ref{c.drama} with Polyakov's prediction \cite{polyakov1975interaction}, we are then urged to find, for any $R>0$ 
examples of translation invariant, ergodic (even strongly mixing)\footnote{Note here that it is not known that infinite volume limits of spin $O(N)$ model are unique, nor ergodic, but it is of course strongly expected.} random environments of conductances which are such that: $i)$ edges of conductances $\leq R$ are exponentially clustering 
and yet $ii)$ the $XY$ model in this field of mostly large conductances is massive. 
This is what we achieve in the theorem below, thus providing a counter-example to the above scenario $B$ from Patrascioiu-Seiler works \cite{patrascioiu1992phase,patrascioiu1993percolation,patrascioiu2002percolation}.

\begin{theorem}\label{th.example}
Fix any inverse temperature $\beta>0$ and any scale $k\geq 1$. 
There exists a one-parameter family of translation-invariant and strongly mixing sub-graphs $(G_m)_{m > 0}$ of  $\Z^2$ on the same probability space such that for\,  $m < m'$ we have $G_m \supset G_{m'}$, $\bigcup_{m>0} G_m = \Z^2$ and
\begin{enumerate}[i)]
\item[i)]The $XY$-model on $G_m$ of inverse temperature $\beta$ exhibits exponential decay. More precisely, there exists a $K=K(m)$ such that for all $x,y \in G_m$
\begin{align*}
	\E^{XY}_{\beta,G_m}\left[\theta(x)\cdot \theta(y) \right]\leq K e^{-m\|x-y\|}.
\end{align*}
\end{enumerate}
Furthermore, for any $m$ sufficiently small:
\begin{enumerate}[i)]\setcounter{enumi}{1}
\item The set $(G_m)^c$ is exponentially clustering.
\item For any $p>p_c(\Z^2)$, i.i.d, edge percolation of intensity $p$ strongly percolates in $G_m$ when $m$ is sufficiently small (in the sense that its complement has exponentially decaying tails).   In other words, a.s. $p_c(G_m) \to p_c(\Z^2)$ as $m\searrow 0$.
\item The Ising model on the graph $G_m$ has long range order at $\beta^{Ising}=1$. 
\end{enumerate}
In particular, this means we obtain ergodic subgraphs $G$ arbitrarily close to $\Z^2$ and for which the ratio $\frac{\beta_c^{BKT}(G)}{\beta^{Ising}_c(G)} \geq \beta$ is as large as one wants. 
\end{theorem}

\begin{remark}\label{}
As in Theorem \ref{th.GFFmGFF}, one could ask more in condition $ii)$: for any $k \geq 1$, the $k-$neighbourhood of $(G_m)^c$ is still exponentially clustering. However, this "lee-room" is morally included already in $iii)$. 

Further, the item $iv)$ follows by a standard FK-percolation argument from the item $iii)$. We still include it as a separate condition, as we find it illuminating.

Finally, we also remark that as $m\to 0$, the rate of exponential decay decreases to 0, however, exponential decay persists, which already is quite surprising.
\end{remark}

\begin{remark}\label{}
Our random sub-graphs will be built using the $m$-massive vectorial GFF (Definition \ref{d.GFF}) with $N = 2$. Item $i)$ will be straightforward from the definition. The interesting part of this Theorem is the fact $i)$ is compatible with $ii)$ and  $iii)$. It may well be that more hands-on examples of such graphs $G_m$ may be built (for example out of Poisson Point Processes of sparser and sparser {\em barriers}), but it would still be challenging to check items $i)$, $ii)$, $iii)$, and more importantly by definition our graphs $G_m$ share the same local quenched environment as the transverse fluctuations in the low temperature classical Heisenberg model (described by Theorem \ref{th.PolyakovEasy} and Corollary \ref{c.drama}). 
\end{remark}

\begin{remark}\label{}
Note that there are continuous spin systems for which exponential decay is known to hold at all temperatures. Let us mention in particular the case where $\sigma_i\in [-1,1]$ in \cite{McBryanSpencer} or the case where spins belong to the semi-spheres $\S^{N-1}_+$ in \cite{bauerschmidt2021geometry}. This latter case is especially remarkable due to its $O(N-1)$ symmetry. Also this example from \cite{bauerschmidt2021geometry} fits well into our analysis as in the case $N=3$, one may also project it to a  $XY$ model in random conductances. In that model it is less clear though that the conductances would indeed be strongly percolating.   
\end{remark}

In Section \ref{s.PS}, we shall extend the spin $O(N)$ model to the cable graph (inspired by the very fruitful extension of the GFF on a graph $\Lambda$ to its cable graph $\tilde \Lambda$ by Titus Lupu in \cite{lupu2016loop}).  This extension has two consequences, first it enables us to reinterpret correlation functions of the spin $O(N)$ model in terms of natural percolation events. 
\begin{theorem}\label{th.corr}
	For any $N \geq 1$, let $x$ and $y$ be two vertices of $\Lambda_n$ and $\theta$ be an $O(N)$-model in $\Lambda_n$ at inverse temperature $\beta$ with free-boundary condition in $\partial \Lambda_n$. We have that
	\begin{align}
		\frac{1}{N} \P(x\stackrel{\tilde \L^1}{\longleftrightarrow} y ) \leq \E\left[\theta(x)\cdot \theta(y) \right] \leq N\, \P(x\stackrel{\tilde \L^1}{\longleftrightarrow } y) 
	\end{align}
\end{theorem}
See Section \ref{s.PS} for all the notations. A similar result was proved in the case $N=3$ in \cite{campbell1998isotropic}. Moreover, when completing this draft we learned that an analogous result has been recently proved in \cite{dubedat2022random}. The two proofs are very different, ours relies on a simple FKG, while \cite{dubedat2022random} relies on a nice resampling argument; furthermore we obtain explicit values for the multiplicative constants.   

The second advantage of this extension to the cable graph is that it allows us to prove a rigorous version of scenario $A$ in Patrascioiu-Seiler works. Namely, consider any translation-invariant infinite volume limit of the spin $O(N)$ model on the cable graph and let $\hat E$ denote the set of edges dual to edges $e$ of $\Z^2$ on which one can find $t\in e \simeq [0,1]$ so that the $N$-component at $t$ vanishes.  Note that if one were to project the spin $O(N)$ model on a spin $O(N-1)$ model, then at low temperature this set $\hat E$ would be made of a vast majority of edges with large conductances. Our analog of scenario $A$ reads as follows. (See also Corollary \ref{pr.PSB}). 
\begin{corollary}\label{th.cable}
For any $N\geq 1$,  if the set $\hat E$ does not percolate (i.e. does not have infinite connected components a.s.), then there exists $C>0$, such that for any $x\neq y$
\begin{align*}\label{}
\EFK{\beta, \Z^2}{}{\sigma(x)\cdot \sigma(y)} \geq \frac 1 {|x-y|^3}\,.
\end{align*}
(when $N\geq 3$, this statement holds for any possible translation invariant infinite volume limit). 
\end{corollary}

\section{Setup and preliminaries}

\subsection{Notations, definitions.}

For all $n\geq 1$, let $\Lambda_n$ be the box $[-n,n]^2\subset \Z^2$. We will set in this case $\p \Lambda_n$ to be the set of vertices in $\Lambda_n$ which are next to $\Lambda_n^c$.  We shall also consider the annulus $\An_{m,n}:= \Lambda_{m} \setminus \Lambda_n \subset \Z^2$ for $m>n$.  (N.B. If we ever take $m\notin \N$, for example $3n/2$, we shall mean its approximation from below, i.e., $\lfloor 3n/2 \rfloor$).

\subsection{Vectorial GFF and its relation to the $O(N)$ model}

\begin{definition}[$N$-vectorial Gaussian free field]\label{d.GFF}
Take $m\geq0$, $\Lambda\subseteq \Z^2$ a finite graph, and set $\partial \Lambda \subseteq \Lambda$ an arbitrary subset. We say that a function $\phi:\Lambda\mapsto \R^N$ is an $m$-massive vectorial GFF with $0$-boundary condition if a.s. $\phi(v)=0$ for all $v\in \partial \Lambda$ and
	\begin{align*}
		\P(d\phi)\propto \exp\left (-\frac{1}{2}\sum_{u\sim u'} \|\phi(u)-\phi(u')\|^2-\frac{m^2}{2}\sum_{u}\|\phi(u)\|^2\right ) \prod_{v\in \Lambda\backslash \partial \Lambda} d\phi(v),
	\end{align*}where here and later sums written as $u \sim u'$ are over undirected edges.
We say that $\phi$ is a GFF if it is a $0$-massive GFF.
\end{definition}

%

\begin{definition}[$O(N)$-model] Take $\Lambda\subseteq \Z^2$ a finite graph. A random function $\theta:\Lambda\mapsto \S^{N-1}$ is said to be an $O(N)$-model with conductances $C:E(\Lambda)\to \R^+$ if
	\begin{align}
		\P(d\theta)\propto \exp\left(\sum_{u\sim u'} C_{uu'} \,\, \theta(u)\cdot\theta(u') \right) \prod_{v\in \Lambda} \lambda^{\S^{N-1}}(d\theta(v)) 
	\end{align}
\end{definition}
The main reason why the massive GFF is related to the $O(N)$-model lies in the following proposition, which is implicit in a lot of the physics literature and that has been recently used in the case where $N=1$ and $\phi$ is non-massive in the works \cite{LW, duminil2020existence}. %
\begin{proposition}\label{pr.anglesGFF}
	Take $m\geq 0$, a finite graph $\Lambda$ with boundary $\partial \Lambda \neq \emptyset$ and let $\phi:\Lambda\mapsto \R^N$ be an $m$-massive 		vectorial GFF. The law of $\theta:=\frac{\phi}{\|\phi\|}\in \S^{N-1}$, the angles of $\phi$, conditionally on $\|\phi\|$ is that of an $O(N)$ model in $\Lambda\backslash \partial \Lambda$ with conductances given by
	\begin{align}
		C_{uu'}= \|\phi(u)\|\|\phi(u')\|, \text{ for any }u\sim u'.
	\end{align}
\end{proposition} 
\begin{proof}
	This follows from the fact that 
	\begin{align*}
		&\P\left(d\theta \mid \|\phi\|=\chi \right)\\
		&\hspace{0.1\textwidth}\propto  \exp\left (-\frac{1}{2}\sum_{u\sim u'} \|\chi\theta(u)-\chi\theta(u')\|^2-\frac{m^2}{2}\sum_{u}\chi(u)^2\right ) \prod_{v\in \Lambda\backslash \partial \Lambda} \lambda^{\S^{n-1}}(d\theta(v))\\
		&\hspace{0.1\textwidth}\propto \exp\left(-\chi(u)\chi(v)(\theta(u)\cdot\theta(v)) \right).
	\end{align*}
	where the proportionality constant depends on $\chi$. This is exactly what we wanted.
\end{proof}

By the exact same proof, we also obtain the following result on the projection of a spin $O(N)$ model down to a spin $O(N-1)$ model.
\begin{proposition}\label{pr.projection}
	Take a finite graph $\Lambda$  and let $\sigma:\Lambda\mapsto \S^{N-1}$ be a spin $O(N)$ model on $\Lambda$ at inverse temperature $\beta$ (with arbitrary boundary conditions). Then, the law of $\frac 1 {\sqrt{1- (\theta^N)^2}}{ (\theta^1,\ldots, \theta^{N-1})} \in \S^{N-2}$, conditionally on $\theta^N$ is that of an $O(N-1)$ model in $\Lambda$ with (random) conductances given by
	\begin{align}
		C_{uu'}= \beta \, \sqrt{1- (\theta^N(u))^2}  \sqrt{1- (\theta^N(u'))^2},\, \text{ for any }u\sim u'.
	\end{align}
\end{proposition}

\subsection{An FKG inequality for a conditioned GFF.}\label{ss.condFKG}

In order to understand the level structure of the GFF, we will need to have estimates on the fluctuations of the GFF near its exploration boundary. To do this, a key tool will be a conditional FKG inequality presented here. See also Lemma 1.3 in \cite{PFFKG} for a statement in a very similar spirit.  

We start by making explicit the law of a GFF conditioned on the event that the field at each point $v$ takes values in some subset $A(v)$ of $\R$.
\begin{definition}\label{d.conditionedGFF}
Let $(A(v))_{v\in \Lambda}$ a family of subsets of $\R$, each made of a finite union of (possibly infinite) intervals. Further, assume that there is a $v_0 \in \Lambda$ such that $A(v_0)$ is bounded. We define $\P^A$ the law of a GFF conditioned on $\phi(v)\in A(v)$ as
	\begin{align*}
		\P^{A}(d\phi)\propto e^{-\frac{1}{2}\sum_{i\sim j} (\phi(i)-\phi(j))^2} \prod_{v}\1_{\phi(v)\in A(v)} d\phi(v).
	\end{align*}
Furthermore, when $A(v) \subseteq \R$ is a discrete set for some $v\in \Lambda_{dis}\subseteq \Lambda$, we extend the definition as follows
\begin{align*}
\P^{A}(d\phi)\propto e^{-\frac{1}{2}\sum_{i\sim j} (\phi(i)-\phi(j))^2} \prod_{v\in \Lambda\backslash \Lambda_{dis}}\1_{\phi(v)\in A(v)} d\phi(v) \prod_{v'\in \Lambda_{dis}}\sum_{a \in A(v)} \delta_{\phi(v')=a}.
\end{align*}
\end{definition}

The following FKG inequality for the conditioned field $\P^A$ might be of independent interest.
\begin{lemma}\label{l.FKG}
	We have that for any  $A$ as in Definition \ref{d.conditionedGFF}, $\P^A$ satisfies the FKG inequality. That is to say, for any increasing function $f,g$
	\begin{align*}
		\E^A\left[ f(\phi)g(\phi)\right]\geq \E^A\left[f(\phi)\right]\E^A\left[g(\phi)\right] 	\end{align*}
\end{lemma}
\begin{proof}
	First, take $\eta>0$ and approximate the law $\P^A$ by the discrete law $\mu^\eta$ that measures functions $\psi \in \calK^{\eta}$. Here $\calK^\eta$ is the set of $\psi:\Lambda\mapsto \R$ such that $\psi(v)\in \eta\Z$ for all $v\in \Lambda\backslash \Lambda_{dis}$ and $\psi(v)\in A(v)$ for all $v\in \Lambda_{dis}$. It is possible to show that $\mu^\eta\to \P^A$ as $\eta \to 0$. Furthermore, $\mu^\eta$ satisfies FKG, as it satisfies the Holley condition \cite{holley1974remarks}	
	\begin{align*}
		\mu^\eta(\psi)\mu^\eta(\psi')\leq \mu^\eta(\psi\vee \psi')\mu^\eta(\psi\wedge \psi').
	\end{align*}
This last inequality follows from the fact that if $\psi,\psi'\in \calK^{\eta}$, the $\psi\vee\psi', \psi\wedge \psi' \in \calK^{\eta}$ and the fact $i,j \in \Lambda$
\begin{align}\label{e.basic_ineq_min_max}
	&(\psi(i)-\psi(j))^2+(\psi'(i)-\psi'(j))^2\\
	\nonumber&\hspace{0.2\textwidth}\geq (\psi\vee \psi'(i)-\psi\vee \psi'(j))^2 + (\psi\wedge \psi'(i)-\psi\wedge \psi'(j))^2.
\end{align}
\end{proof}

\subsection{Bounding the size of local sets using their conditional expectation.}

The goal of this subsection is to introduce a result saying that if the variance of $\phi_A(0)$ is small for a local set $A$, then the set itself cannot be that big. Similar techniques to control the size of local sets have been used in \cite{Aru, ASW, ALS1}. 

We start by recalling the definition of a local set.

\begin{definition}[Local set]\label{d.LocalSet}
We call a random subset $A\subseteq \Lambda$ a local set, if there exist two random functions, $\phi_A$ and $\phi^A$, on $\Lambda$ with $0$-boundary condition in $\partial \Lambda$, such that that conditionally on $A = A_0$ we have
	\begin{enumerate}
		\item $\phi=\phi_A+\phi^A$.
		\item $\phi_A$ is harmonic in $\Lambda\backslash A_0$.
		\item $\phi^A$ is a GFF in $\Lambda$ with $0$-boundary condition in $A_0\cup \partial \Lambda$.
	\end{enumerate}
\end{definition}

\begin{remark}\label{r.optional}
	A sufficient condition for a random set $A\subseteq \Lambda$ to be a local set is that $A$ is an optional set that is to say for any deterministic $C\subseteq \Lambda$
	\begin{align*}
		\{A\subseteq C\} \in \sigma(\phi(x):x \in C).
	\end{align*}
\end{remark}

The statement alluded to above can be then formalized as follows.	

\begin{proposition} \label{pr.small_harmonic_function}
	There exists a function $C:\R^+\mapsto \R^+$ such that the following is true uniformly in $n$. Let $\phi$ be a GFF in $\Lambda_n$ and $A$ be a local set such that $A\cup \partial \Lambda_n$ is connected and $\E\left[\phi_A^2(0) \right]\leq \delta$, then
	\begin{align}
		\P(A\cap \Lambda_{n(1-\epsilon)}\neq \emptyset)\leq \frac{\delta}{C(\epsilon)}
	\end{align}
\end{proposition}

\begin{proof}
	Note that there exists a function $C:\R^+\mapsto \R^+$ such that uniformly on all $n\in \N$ and uniformly on all $A$ s.t.  $A\cap \Lambda_{n(1-\epsilon)}\neq \emptyset$, we have 
	\begin{equation}
		G_{\Lambda_n \backslash A,\partial \Lambda_n}(x,x)\leq G_{\Lambda_n,\partial \Lambda_n}(x,x)-C(\epsilon).
	\end{equation}
Now, we define $p:=\P(A\cap \Lambda_{n(1-\epsilon)}\neq \emptyset)$. 
By the domain Markov property for the GFF,  we have 
\begin{align*}
\Eb{\phi^2(0)} & = \Eb{\phi_A(0)^2 + \phi^A(0)^2} \\
& \leq \delta + \Eb{G_{\Lambda_n \setminus A, \p \Lambda_n}(0,0)} \\
& \leq \delta + p(G_{\Lambda_n,\partial \Lambda_n}(0,0)-C(\epsilon)) + (1-p) G_{\Lambda_n,\partial \Lambda_n}(0,0)\,,
\end{align*}
where we used the fact that $\Eb{\phi^A(0)^2 \md \F_A} \leq G_{\Lambda_n,\partial \Lambda_n}(0,0)$ a.s. Using now the fact that $\Eb{\phi^2(0)}=G_{\Lambda_n,\partial \Lambda_n}(0,0)$, we deduce 
\begin{align*}
	p\leq \frac{\delta}{C(\epsilon)}.
\end{align*}	
\end{proof}

\subsection{Loop soups and isomorphism theorems on subgraphs of $\Z^2$ and on the associated metric graphs.}

We denote a vectorial random walk loop-soup on a subgraph of $\Z^2$ by $\L = (\L^1, \dots, \L^N)$ where each of $\L^i$ is a random walk loop soup at the critical intensity $1/2$. We denote the vector of the local times by $(L^1, \dots, L^N)$ and the total local time over the coordinates  by $L = L^1+\dots+L^N$. There is also a natural $m$-massive version of the loop soup, denoted $\L^{(m)}$. See e.g. \cite{Sznitmantopics,WP} for precise definitions of loop soups and their local times. As we will only work at intensity $1/2$, we will often omit this detail and refer just to the random walk loop soup.

We will be using at several places the isomorphism theorem between the local times of the random walk loop soup and the Gaussian free field. We again refer the reader to \cite{WP} or \cite{LeJan2011Loops}, but state a version of it here for the convenience of the reader.

\begin{theorem}[Isomorphism theorem for the massive and non-massive GFF]\label{th.isom1}
Consider a vectorial random walk loop soup $\L$ on a subgraph of $\Z^2$ with non-zero boundary and zero boundary conditions on that boundary. Then $(L^1, \dots, L^N)$ has the same law as $((\phi^1)^2, \dots, (\phi^N)^2)$ where $\phi$ is an $N$-vectorial GFF on the same graph with the same boundary conditions.

The same holds for the massive versions with either free or zero boundary conditions for both the loop soup and the field. 
\end{theorem}

We will at some point also work with the metric graph: for a subgraph of $\Z^2$ its metric graph version can be just seen as the set $V \cup E \subseteq \R^2$ with the subset metric induced from the usual Euclidean metric on $\R^2$. One can define both (vector-valued) random walk loop soups $\tilde \L$ and GFFs $\tilde \phi$ on the metric graphs \cite{lupu2016loop, zhai}. See again \cite{WP} for more details.

\begin{theorem}[Signed isomorphism theorem on the metric graph \cite{lupu2016loop}]
\label{th.isom2}
Consider a vectorial random walk loop soup $\tilde \L$ on the metric version of a subgraph of $\Z^2$ with non-zero boundary and zero boundary conditions on that boundary. Then $(\tilde L^1, \dots, \tilde L^N)$ has the same law as $((\tilde \phi^1)^2, \dots, (\tilde \phi^N)^2)$ where $\tilde \phi$ is an $N$-vectorial GFF on the same graph with the same boundary conditions.

Further, if we let $s^i$ to be equal to a $\pm 1$ valued random function defined by sampling an independent Rademacher r.v. for each connected component of $\{\tilde L^i \neq 0 \}$, we have that $(s^1 \sqrt{\tilde L^1}, \dots, s^N\sqrt{\tilde L^N})$ is equal in law to $(\tilde \phi^1, \dots, \tilde \phi^N)$.

The same holds for the massive versions with either free or zero boundary conditions for both the loop soup and the field. 
\end{theorem}

Finally, we also make use of the following standard coupling between massive and non-massive random walk loop soups on $\Lambda_n$ that stems just from the coupling between massive and non-massive random walks (see e.g. Proposition 3.2 in \cite{camiamass}).
\begin{lemma}[Coupling of massive and non-massive loop soups]\label{l.coupling}
	Fix	 $n\in \N$ and $\epsilon>0$, there exists $m_0=m_0(\eps)$ such that for all $m<m_0$ the following is true. There exists a coupling between $\mathcal L^{(0)}$ a massless-random walk loop soup and $\mathcal L^{(m)}$ an $m$-mass random walk loop soup, both defined on $\Lambda_n$ with zero boundary conditions, such that with probability $1 -\eps$ we have that $\mathcal L^{(0)}=\mathcal L^{(m)}$.
\end{lemma}

\begin{remark}\label{}
If one would want to make our results more quantitative, it is easy to check that in this Lemma, one may choose the mass $m_0$ to be of order $\eps^{1/2} n^{-1}$.
\end{remark}

\subsection{A percolation estimate.}

To finish preliminaries, we recall a classical result on the size of open clusters for subcritical Bernoulli percolation on $\Z^2$.	

\begin{theorem}[\cite{aizenman1987sharpness,menshikov1986coincidence,duminil2017new}]\label{th.sharp}
For any $0\leq p<p_c$ there exists a $C=C(p)$ and $\alpha:=\alpha(p,C)>0$ such that the probability of the  event $(\mathcal C_C)^c$ is smaller than $Cn^{-\alpha}$ as $\Lambda_n\nearrow \Z^2$.
\end{theorem}
\section{Discrete two-valued sets are not macroscopic}\label{s.AR}

The aim of this section is to prove Theorem \ref{th.ExitSet}, restated here for the convenience of the reader.
\begin{theorem}\label{th.ExitSet2}
Let $N\geq 2$. For any $\eps>0$ and any $R>0, k\geq 1$, if $\phi : \Lambda_n=[-n,n]^2 \to \R^N$ is a vectorial GFF with zero-boundary conditions on $\Lambda_n$ and if $A_{R,k}$ is the exit set defined in~\eqref{e.ARk}, then 
\begin{align*}\label{}
\lim_{n \to \infty}  \Pb{A_{R,k} \cap \Lambda_{(1-\eps) n}} = 0\,. 
\end{align*}
More quantitatively, we may allow the thresholds $R$ and $k$ to depend on the scale $n$:
\begin{align*}\label{}
\lim_{n \to \infty}  \Pb{A_{R(n),k(n)} \cap \Lambda_{(1-\eps) n}} = 0\,,
\end{align*}
as long as $(R(n) + \log k(n))^2 = o(\log n)$.
\end{theorem}

Recall also the notation of $A_{R,k}$: this is the set of all vertices of $\Lambda_n$ which can be reached from the boundary by a path where the norm of the GFF remains less than $\R$, and where a path can jump to at most graph distance $k$ on every step. Formally, $A_{R,k}$ is equal to 
\begin{align*}\label{e.ARk2}
\left \{  x \in \Lambda_n:  (\exists x_0\in \partial \Lambda_n, \ldots, x_m=x) (\forall 0\leq i\leq m-1)\,   d(x_i,x_{i+1}) \leq k, \|\phi(x_i)\|\leq R \right \}
\end{align*}
It is easy to see that these sets are optional sets of the GFF as defined in Remark \ref{r.optional} and thus give rise to a Markovian decomposition in the sense of Definition \ref{d.LocalSet}.

For technical reasons, it will be also useful to consider the optional set $\hat A_{R,k}$ where we add the additional constraint that all vertices `explored' are at most at distance $n/2$ from the boundary. More precisely $x\in \hat A_{R,k}$ if there exist a path  $\{ x_0\in \partial \Lambda_n, \ldots, x_m=x\}\subseteq \Lambda_n\backslash \Lambda_{n/2}$ such that:
\begin{align}
 d(x_i,x_{i+1}) \leq k \text{ and } \|\phi(x_i)\|\leq R, \,\, \ \ \ \ \forall 0 \leq i \leq m-1.
\end{align}

The theorem is proved in several steps:
\begin{itemize}
\item First, in Subsection \ref{subseq:flct} we obtain bounds on the fluctuations of the GFF near the boundary of our exit set.
\item Second, in Subsection \ref{subseq:MW} we develop a Mermin-Wagner type argument to see that the angles on the boundary of the exit set mix well.
\item Finally, in Subsection \ref{subseq:fnl} we use this argument to obtain that the harmonic function of $\hat \phi_{A_{R,k}}$ has negligible variance at the origin.
\end{itemize}

This final step results in the following proposition:

\begin{proposition}\label{p.small_harmonic_function_for_explorations}
	Let $\hat A_{R,k}$ be the exit set of height $R$ on the graph with jumps allowed to distance $k$ and stopped when entering $\Lambda_{n/2}$. We have that 
	\begin{align*}
		\E\left[\phi_{\hat A_{R,k}}^2(0) \right] \leq C\frac{(R+\log(k+2))^6}{\log n}.
	\end{align*}
\end{proposition}

Theorem \ref{th.ExitSet2} follows easily by plugging this estimate into Proposition \ref{pr.small_harmonic_function}. 

\begin{proof}[Proof of Theorem \ref{th.ExitSet2}]
Notice that it suffices to prove the theorem with $\hat A_{R,k}$ instead of $A_{R,k}$, which for the ease of notation we just denote by $A$. 
By Proposition \ref{p.small_harmonic_function_for_explorations}, we see that for any $\delta>0$, there exists $n$ sufficiently big so that $\E\left[\phi_{A}^2(0) \right]\leq \delta$. We can now use Proposition \ref{pr.small_harmonic_function} to see that for any $\epsilon,\delta>0$
	\begin{align}
		\limsup_{n\to \infty}\P[A\cap \Lambda_{n(1-\epsilon)}\neq \emptyset]\leq \frac{\delta}{f(\epsilon)},
	\end{align} 
from where we conclude by letting $\delta \to 0$.
\end{proof}

\subsection{Fluctuations of the field near an exploration boundary.}\label{subseq:flct}

The goal of this subsection is to understand the fluctuations of the free field close to the explored region. This is summarised in the following proposition.
\begin{proposition}\label{pr.fluctuations}
	Let $\Lambda$ be any finite subset of $\Z^2$ and $\phi: \Lambda \to \R^N$ be an $N$-vectorial GFF with $0$ boundary condition on $V_{\partial}\subseteq \Lambda$ and $a:\Lambda \to \R^+$ a function upper bounded by $a_\infty \in \R^+$. Let  $V_{\leq}$ and $V_{>}$ be non-empty subsets of $\Lambda$  and let $\calK$ be the event that $\|\phi(v)\| \leq a(v)$ on $V_{\leq}$ and $\|\phi(v)\|>a(v)$ on $V_{>}$. 
	
Then there exists a constants $C=C(N)<\infty$ which does not depend on $\Lambda$ nor $V_\p,V_{\leq},V_{>},$ such that
	\begin{align}\label{}
		\sup_{d(v,V_{\leq })\leq \rho}\Eb{e^{\lambda \|\phi(v)\|}\md \calK} \leq Ce^{\lambda N a_\infty + \frac{2N}{\sqrt{2\pi}}\lambda \log (\rho+1)+\frac{N}{2\pi}\frac{\lambda^2}{2}\log(\rho+1)}, \text{ for all }\lambda>0,\, \rho\geq 0.
\end{align}
In particular, for $p \in \N$ there exists $K_p=K(p, N) > 0$ such that 
\begin{align}\label{eq:2nd}
	\sup_{d(v,V_{\leq })\leq \rho}\E\left[\|\phi(v)\|^{2p} \md \calK \right]\leq K_p(a_\infty+\log(\rho+2))^{2p}. 
\end{align}
\end{proposition}

\begin{remark}\label{}
This result is similar in spirit to Lemma 3.1 in \cite{schramm2009contour} and generalizes parts of it, by making the estimates more quantitative. Our proof is not inductive as in \cite{schramm2009contour} and relies on the conditional FKG of the GFF, introduced in Subsection \ref{ss.condFKG}. 
\end{remark}

The proof relies on two lemmas. First, using the FKG inequality for the GFF conditioned on $\phi(v)\in A(v)$ we reduce the proposition to the case where $V_{\leq}=\{v_0\}$ and $V_{>}=\Lambda\backslash \{v_0\}$, and $\phi(v)>a$ for all $v\in V_{>}$.

\begin{lemma} \label{l.ineq1}Let us work in the context of Proposition \ref{pr.fluctuations}, but where $\phi$ is the usual scalar GFF. Fix $v\in \Lambda$ and choose $v_0\in V_{\leq}$ that minimizes $d(v, V_{\leq})$. Let $\Lambda' \subseteq \Z^2$ be a graph containing $\Lambda$ (and its boundary) and with $\partial \Lambda'$ possibly empty. Define $\calK^{\max}$ as the (zero probability) event where $\phi_1(v_0)=a_\infty$ and $\phi_1(v)> a_\infty$ for all $v\in \Lambda'\backslash \{v_0\}$. Then
	\begin{align}\label{e.ik+kmax}
		\E_{\Lambda}\left[e^{\lambda |\phi(v)|}\mid \calK \right] \leq 2\E_{\Lambda'}\left[e^{\lambda \phi(v)}\mid \calK^{\max} \right], \text{ for all } \lambda>0, v\in \Lambda. 
	\end{align}
	Here the GFF in $\Lambda'$ has Dirichlet boundary conditions only at $v_0$.
\end{lemma}

Second, we bound the RHS in this lemma by finding a function $F$ such that the event $\calK^+$ holds for $\phi+F$ with positive probability. One can show that taking $F(u)$ roughly as $c\ln(\|u-v_0\|)$ does the job. Using again the FKG inequality, Cameron-Martin theorem for the GFF and the explicit form $F$, we then obtain the desired estimate. This lemma is inspired by the study of entropic repulsion of the GFF (\cite{bolthausen1995entropic}).

\begin{lemma} \label{l.ineq2}Let $v \in \Z^2 \setminus \{0\}$ and for $R > d(v, 0)$ let $\Lambda':=B(0,R)\cap\Z^2$ be an Euclidean ball. Define $\calK^{\max}$ to be as in the previous lemma with $v_0 =0$. Then, there exists a constant $C>0$ that does not depend on $R$ such that
	\begin{align}
		\E_{\Lambda'}\left[e^{\lambda\phi(v)}\mid \calK^{\max} \right]\leq Ce^{\lambda a_\infty} e^{\frac{2}{\sqrt{2\pi}}\lambda \ln (\|v\|+1) + \frac{\lambda^2\ln(\|v\|+1)}{4\pi }} .
	\end{align}
	Here the GFF on $\Lambda'$ has boundary conditions only at the origin.
\end{lemma}

Given these lemmas, the proof of Proposition \ref{pr.fluctuations} is direct.

\begin{proof}[Proof of Proposition \ref{pr.fluctuations}]
	Take $v$ that is at Euclidean distance smaller than $\rho$ of $V_{\leq}$.
	\begin{align}\label{e.prop3.1firststep}
		\Eb{e^{\lambda \|\phi(v)\|}\md \calK} \leq  \E\left[ e^{\lambda\sum_i|\phi_i(v)|}\mid \calK\right].
	\end{align}
Let us now condition on $|\phi_i|$ for $i\in \{1,..,N-1\}$ and $\calK$, and call this conditioning $\calK_N$.  The law of $\phi_{N}$ under the conditioning $\calK_N$ is equal to the following: we have the same set of vertices $V_{\leq}$ and $V_{>}$ as in the $N=1$ case except now, $a(v):=\sqrt{a(v)-\sum_{i=1}^{N-1}\phi_i^2(v)}\leq a_\infty$ (note that the value in the square root is always positive for $\phi\in \calK$ and it is measurable w.r.t $\calK_N$). Thus, \eqref{e.prop3.1firststep} is upper bounded by
\begin{align*}
	&\E\left[e^{\lambda\sum_{i=1}^{N-1}\|\phi_i(v)|}\E\left[e^{\lambda \|\phi_N(v) \|} \mid \calK_N\right] \mid \calK  \right]\\
	&\hspace{0.1\textwidth}\leq Ce^{\lambda  a_\infty + \frac{2}{\sqrt{2\pi}}\lambda \log (\rho+1)+\frac{1}{2\pi}\frac{\lambda^2}{2}\log(\rho+1)}\E\left[e^{\lambda\sum_{i=1}^{N-1}\|\phi_i(v)\|} \mid \calK \right],
\end{align*}
where we used Lemmas \ref{l.ineq1} and \ref{l.ineq2} for $\E\left[e^{\lambda |\phi_N(v) |} \mid \calK_N\right]$. Iterating this procedure we obtain the desired result.

\end{proof}

It remains to prove the two lemmas.

\begin{proof}[Proof of Lemma \ref{l.ineq1}]
First, note that we just need to prove that
\begin{align*}
\E_{\Lambda}\left[e^{\lambda \phi(v)}\mid \calK \right] \leq \E_{\Lambda'}\left[e^{\lambda \phi(v)}\mid \calK^{\max} \right], \text{ for all } \lambda>0, v\in \Lambda. 
\end{align*}
This is because 
\begin{align*}
\E_{\Lambda}\left[e^{\lambda |\phi(v)|}\mid \calK \right]\leq \E_{\Lambda}\left[e^{\lambda \phi(v)}\mid \calK \right] +\E_{\Lambda}\left[e^{\lambda (-\phi(v))}\mid \calK \right],
\end{align*}
and the fact that $-\phi$ has the same law as $\phi$.

   We start by showing the result when the graph $\Lambda'=\Lambda$ and $\partial \Lambda ' = \partial \Lambda$.
	Define $\calK_1$ to be an auxiliary event where $|\phi(v)| > a(v)$ on $V_{>}$, $\phi(v) \geq -a(v)$ on $V_{\leq}\setminus\{v_0\}$ and $\phi(v_0) \in [-a(v_0), a_\infty]$. Notice that $\calK \subseteq \calK_1$. Note that on $\calK_1$, the event $\calK$ is decreasing. As a consequence, we have from the FKG inequality for $\P(\cdot \mid \calK_1)$ that
	\begin{align*}
\E_{\Lambda}\left[e^{\lambda \phi(v)}\mid \calK \right] \leq \E_{\Lambda}\left[e^{\lambda \phi(v)}\mid \calK_1 \right].
	\end{align*}
	Now, for all $\eps > 0$ small, define $\calK^{\max, \eps}$ as the (positive probability) event where $\phi(v_0) \in [a_\infty - \eps, a_\infty]$ and $\phi(v)> a_\infty$ for all $v\in \Lambda \backslash V_\partial$ with $v \neq v_0$. Notice that $\calK^{\max, \eps} \subseteq \calK_1$ and, on $\calK_1$, the event $\calK^{\max, \eps}$ is increasing. Again from the FKG inequality for $\P(\cdot \mid \calK_1)$ we obtain
	$$\E_{\Lambda}\left[e^{\lambda \phi(v)}\mid \calK \right] \leq \E_{\Lambda}\left[e^{\lambda \phi(v)}\mid \calK^{\max, \eps} \right].$$
	Letting $\eps \to 0$ gives the result on the original graph $\Lambda$ with its zero boundary.
	
	To extend the result to $\Lambda' \supset \Lambda$ with possibly no boundary, we use the same strategy, defining
	\begin{itemize}
	\item The event $\calK_0$ as $\calK$ intersected with the (zero probability) event where $\phi(v) = 0$ on all $v \in \partial \Lambda \cup \Lambda'\setminus \Lambda$.
	\item An auxiliary event $\calK_2$ defined by $\phi(v_0) = a_\infty$, $\phi(v) > a_\infty$ for all $v\in \Lambda \backslash V_\partial$ and $\phi(v)\geq 0$ for all $v \in \partial \Lambda \cup \Lambda'\setminus \Lambda$.
	\end{itemize}
	Notice that again $\calK_0 \subseteq \calK_2$ and $\calK^{\max} \subseteq \calK_2$ and the first event is decreasing and the second increasing  w.r.t. $\P(\cdot \mid \calK_2)$. Conditional FKG (by going through positive probability events as above) gives us
	$$\E_{\Lambda'}\left[e^{\lambda \phi(v)}\mid \calK_0 \right] \leq \E_{\Lambda'}\left[e^{\lambda \phi(v)}\mid \calK_2 \right] \leq \E_{\Lambda'}\left[e^{\lambda \phi(v)}\mid \calK^{\max} \right].$$
	As $\P_{\Lambda'}(\cdot \mid \calK_0) = \P(\cdot \mid \calK^{\max})$, we obtain the result
\end{proof}

\begin{proof}[Proof of Lemma \ref{l.ineq2}]
As the GFF with boundary condition $\phi(0) = a_\infty$ is equal to the GFF with zero boundary at $0$ plus $a_\infty$, we can assume that  $a_\infty=0$. We use the measure $\E_{\Z^2}$ for the GFF on $\Lambda=\Z^2$   and $\E_{\Lambda'}$ for the GFF in $\Lambda'$ with zero boundary conditions at the origin (i.e. $\partial \Lambda = \partial \Lambda' =\{0\}$).

Now consider 
\begin{align*}
	F(v) := 2\sqrt{2\pi}\E_{\Z^2}[\phi(v)^2] = \sqrt{2\pi}G_{\Z^2,\{0\}}(v,v).
\end{align*}
It is known that $F(v)$ is harmonic off the origin (see Proposition 4.4.1 and 4.4.2 in \cite{LawlerLimic2010RW}) and moreover from Theorem 4.44 in \cite{LawlerLimic2010RW} we have that\footnote{Note that the constant changes, due to the fact that we are using the graph Laplacian and in \cite{LawlerLimic2010RW} they use the random walk Laplacian.}
\begin{equation}\label{eq:green}
G_{\Z^2, \{0\}}(v,v) = \frac{1}{\pi} \log \|v\| + C + O(\|v\|^{-2}).
\end{equation}
We claim that 
\begin{enumerate}
\item the event $\calK^{\max}$ holds under $\phi + F$ with positive probability in $\Z^2\setminus\{0\}$
\item when we denote by $\P^F$ the probability measure for $\phi + F$ then $$\E_{\Lambda'}\left[e^{\lambda\phi(v)}\mid \calK^{\max} \right] \leq \frac{e^{\lambda F(v)} \E_{\Lambda'}\left[e^{\lambda \phi(v)} \right]}{\P_{\Lambda'}^F(\calK^{\max})}$$
\end{enumerate}
Combining these two points and calculating the Laplace transform $\E\left[e^{\lambda \phi(v)} \right]$, we obtain the proposition. It thus remains to prove the two claims.

To prove the first claim, we use the union bound and the standard Gaussian bound $\P(\mathcal{N}(0,1) \leq x) < \exp(-x^2/2)$
	\begin{align*}
		\P_{\Lambda'}^F(\phi \notin \calK^{\max})&\leq \sum_{v\in \Lambda'\backslash \{0\}} \P\left[\phi(v)\leq-F(v)  \right]\\
		& \leq \sum_{v\in \Lambda'\backslash \{0\}}\exp\left( -\frac{F(v)^2}{2 \E[\phi(v)^2]}\right) \leq C\sum_{v\in \Lambda' \backslash \{0\}} e^{-2\ln n}<\infty.
	\end{align*}
This is enough as it shows that $\P^F_{\Lambda'}(\phi(v)>0, \forall \|v\|\geq R)$ goes to $0$ as $R\to \infty$ and 
	\begin{align*}
\P^F_{\Lambda'}( \phi\in \calK^{\max}\mid \phi(v)>0, \forall \|v\|\geq R)
	\end{align*} is always positive.

For the second claim, we want to use again the conditional FKG. To do this we define
	\begin{align*}
		\frac{d\P_{\Lambda'}^F(\cdot\mid \calK^{\max})}{d\P_{\Lambda'}(\cdot\mid\calK^{\max})}&= \frac{e^{\sum_{i\sim j}(\phi(i)-\phi(j))(F(i)-F(j)) }1_{\calK^{max}}} {\E_{\Lambda'}\left[ e^{\sum_{i\sim j}\phi(i)-\phi(j))(F(i)-F(j)}\mid \calK^{\max}\right] }\\
		&= \frac{e^{\sum_{v}\phi(v)(-\Delta)F(v)}1_{\calK^{max}}}{\E_{\Lambda'}\left[e^{\sum_{v}\phi(v)(-\Delta)F(v) }\mid\calK^{\max}\right] },
	\end{align*}
	where we used integration by parts, taking $\Delta$ inside the graph $\Lambda'$.
But now $F(v)$ is harmonic in $\Z^2 \backslash \{0\}$, $\phi(0) = 0$. Finally one needs to check that $-\Delta F(v) \geq 0$ for all $v$ on the boundary of $\Lambda'$, as long as $R$ is large enough. This follows when one shows that $F(v)$ is increasing along each edge that increases its Euclidean distance. This in turn seems true, but not quite obvious to argue, so technically it is more convenient to consider $\tilde F(v)$, which is just given by the harmonic function inside $B(0,R) \setminus \{0\}$, whose boundary values are $\sqrt{2/\pi}\log \|v\|$ on the outer boundary, and fixed to $0$ at $0$. Then $\tilde F(v)$ is harmonic in the interior $B(0,R) \setminus \{0\}$ and clearly satisfies the required monotonicity property on the boundary. Moreover, by \eqref{eq:green} and maximum principle, $\|F(v) - \tilde F(v)\| = O(1)$.

We conclude that $e^{\sum_{i\sim j}(\phi(i)-\phi(j))(\tilde F(i)- \tilde F(j)) }$ is increasing w.r.t. $\phi$ and thus by FKG for the measure $\P_{\Lambda'}(\cdot \mid \calK^{\max})$, we obtain that
\begin{align*}
	\E_{\Lambda'}\left[e^{\lambda \phi(v)} \mid \calK^{\max}\right]\leq \E_{\Lambda'}^{\tilde F}\left[e^{\lambda \phi(v)} \mid \calK^{\max} \right] = \frac{\E^{\tilde F}_{\Lambda'}\left[ e^{\lambda \phi(v)}1_{\calK^{\max}}\right]}{\P_{\Lambda'}^F(\calK^{\max})}. 
\end{align*}
We now conclude by omitting the indicator function in the numerator, using the usual Cameron-Martin theorem and the comparison of $F$ and $\tilde F$.

\end{proof}

\subsection{Mermin-Wagner theorem for the exit sets of the GFF.}\label{subseq:MW}
The goal of this section is to study the correlation of the angles of the GFF on the boundary of its level set exploration. We show a Mermin-Wagner type of result \cite{mermin1966absence, mermin1967absence}, proving a quantitative decay of the correlations. We will state and prove it here for the $2$-vectorial GFF, but as in the case of $O(N)$ models, it can be then generalized to $N \geq 3$.

\begin{proposition}[Mermin-Wagner for the GFF] \label{pr.MW}
Let $\phi$ be a $2$-vectorial GFF in a graph $\Lambda_n$, $\theta:=\frac{\phi}{\|\phi\|}$ be the angles of $\phi\in \R^2$ and let $A_{R,k}$ be a level set exploration of $\phi$ and $A$ some subset of $\Lambda_n$. 

Then there exists a constant $K$ such that for any $x, y \in A$
	\begin{align}\label{eq:mw}
			\E\left[\theta(x)\cdot \theta(y) \mid A_{R,k}=A\right]\leq \left (K\frac{R+\ln(k+2)}{\sqrt{\ln(d(x,y)+1)}}\right )\wedge 1,
		\end{align}
\end{proposition} 

The proof mimics the one for the $XY$ model (see e.g. \cite{velenikBook} Chapter 9), by making use of the fact that conditionally on the norm of the GFF, we obtain an $XY$ model with interaction strengths depending on the norm of the field. As we control the norm of the field near the exploration boundary by the last subsection, we can conclude. 

\begin{proof}
We start by describing what is the conditional law of $(\theta(v))_{v\in A}$ given that the exploration set $A_{R,k} = A$. By the Markov property we can write $\phi=\phi_A+\phi^A$. Here, conditionally on $A_{R,k} = A$, $\phi^A$ is a GFF in $D\backslash A$ independent of $\phi_A$ and $\phi_A$ is equal to $\phi$ on $A$, $0$ in $\partial \Lambda$ and is harmonic on $\Lambda_n\backslash (A\cup \partial \Lambda_n)$. Furthermore, 
	\begin{align*}
		\P\left[d\phi_A \right]\propto e^{-\frac{1}{2}\sum_{u \sim u'} \|\phi_A(u)-\phi_A(u')\|^2} \prod_{v\in A}d\phi(v).
	\end{align*}
 Now, for any $v\in \Lambda_n$
\begin{align}
	\phi_A(v)= \sum_{s\in A} \phi_A(y)p^0_{v\to s},
\end{align}
where $p^0_{v\to s}$ is the probability that a random walk started from $v$ first hits $A\cup \partial \Lambda_n$ at the point $s$ (if $v\in A$, this value is $\delta_{s=v}$). Furthermore, note that $\sum_{u \sim u'} \|\phi_A(u)-\phi_A(u')\|^2= \sum_{s\in A} \phi_A(-\Delta\phi_A)(s)$. Thus,

\begin{align}\label{e.energyphi_A}
\sum_{u \sim u'} \|\phi_A(u)-\phi_A(u')\|^2 = \frac{1}{2}\sum_{(s,t) \in A\times A} d_s\|\phi_A(s)-\phi_A(t)\|^2p^1_{s\to t},
\end{align}
where we recall that by our convention the first sum is over unordered pairs and $d_s$ is the degree of $s$.
Here,  $p^{1}_{s\to t}$ is the probability that a random walk started from $s$ at any time $T\geq 1$ enters $A\cup \partial \Lambda_n$ at the point $t$.

Denoting by $r = \|\phi\|$ and its restriction to $A$ as $r_A$, we see that the law of $\theta$ restricted to $A$ given that $A_{R,k} = A$ and $(\|\phi_A(s)\|)_{s \in A}$ can be described by
\begin{align*}
 \P^{r_A}(d\theta)\propto e^{-\frac{1}{2}H(\theta) }\prod_{s\in A} d\theta(s),
\end{align*}
where 
\begin{align*}
	H(\theta)=\frac{1}{2}\sum_{s,t \in A} d_s\|r_A\theta(s)-r_A\theta(t)\|^2p^1_{s\to t}.
\end{align*}

	We now aim to show that this law changes minimally when one rotates all angles $\theta$ by a well chosen angle $(S(v))_{v \in \Lambda_n}$. 
	In this respect, for $x \neq y\in A$ and $v \in \Lambda_n$ define
	\begin{align}
		S(v) := S_{x,y}(v) :=\pi\frac{ G_{\Lambda_n,\{x\}}(y,v)}{G_{\Lambda_n,\{x\}}(y,y)}.
	\end{align}
	Here $G_{\Lambda_n,\{x\}}$ is the Green's function with zero boundary conditions at the vertex $x$ only. We can then calculate
	\begin{align}\label{eq:Sbound}
		\sum_{v\sim v'} (S(v)-S(v'))^2 = \sum_{v} S(v)(-\Delta)S(v) = \frac{\pi}{G_{\Lambda_n,\{x\}}(y,y)} \leq \frac{C}{\ln(\|x-y\|)},
	\end{align}
	where the last inequality follows from \eqref{eq:green} and the fact that the zero boundary Green's function of $\Lambda_n\setminus\{x\}$ converges to that of $\Z^2 \setminus \{x\}$. 
	
	Now denote by $\hat\P^{r_A}$ the law of $ e^{iS}\theta$, i.e. the angles $\theta$ rotated by $e^{iS}$, conditionally on $A_{R,k} = A$ and $\|\phi_A(s)\| = r_s$ for all $s \in A$. Its Radon-Nikodym derivative w.r.t. the original law on angles $\P^{\theta}$ under the same conditioning is given by
	\begin{align}
		\frac{d\hat\P^{r_A}}{d\P^{r_A}}= e^{\frac{1}{2}(H(\theta)- H(e^{iS} \theta))}.
	\end{align}
We can now use Pinsker's inequality (Lemma B.67 of \cite{velenikBook}) to obtain
	\begin{align*}
		2\E^{r_A}\left[\theta(x)\cdot \theta(y) \right] =
		|\E^{r_A}\left[\theta(x)\cdot \theta(y) \right] - \hat \E^{r_A}\left[\theta(x)\cdot \theta(y) \right]| \leq \sqrt{2 \mathrm h(\P^{r_A}\mid \hat \P^{r_A})},
	\end{align*}
where $\mathrm h(\P^{r_A}\mid \hat \P^{r_A})$ is the relative entropy between $\P^{r_A}$ and $ \hat \P^{r_A}$. We have that $\mathrm h(\P^{r_A}\mid \hat \P^{r_A})$ is equal to $1/4$ times 

\begin{align*}
 \E^{r_A}\left[ \sum_{s, t \in A \times A}d_sp^1_{s\to t}(\|e^{iS(s)}r_s\theta(s)- e^{iS(t)}r_t\theta(t)\|^2)-\|r_s\theta(s)- r_t\theta(t)\|^2 \right].
\end{align*}
By opening the squares we can write each summand  as
$$d_sp^1_{s\to t}2r_sr_t\left[\cos(\theta(s)-\theta(t))(1-\cos(S(s) - S(t)))+\sin(\theta(s)-\theta(t))\sin(S(s)-S(t))\right].$$
Now the $\sin$ term will cancel after taking expectations as $\theta(s) \sim -\theta(s)$ in law under $\E^{r_A}$.
Thus using $2- 2\cos(x) \leq x^2$ we obtain that $2(\E^{r_A}\left[\theta(x)\cdot \theta(y) \right])^2$ can be bounded by
\begin{equation}\label{eq:condbnd}
\mathrm h(\P^{r_A}\mid \hat \P^{r_A})\leq \frac{1}{4}\E^{r_A}\left[\sum_{(s,t) \in A \times A}d_s p^1_{s\to t} \|S(s)-S(t)\|^2	r_sr_t\right]
\end{equation}
and hence for $x \neq y$
$$\E^{A}\left[\theta(x)\cdot \theta(y) \right]=\E^A\left[\E^{r_A}\left [ \theta(x)\cdot \theta(y) \right ]\right]$$
can be upper bounded by
$$\frac{1}{2\sqrt 2}\sqrt{\sum_{s,t \in A} d_sp^1_{s\to t}\|S(s)-S(t)\|^2\sqrt{\E\left[|\phi_A(s)|^2\mid A_{R,k} = A \right]  \E\left[|\phi_A(t)|^2\mid A_{R,k} = A \right]} }.$$
By using \eqref{eq:2nd} we can upper bound $$\sqrt{\E\left[|\phi_A(s)|^2\mid A_{R,k} = A \right]  \E\left[|\phi_A(t)|^2\mid A_{R,k} = A \right]} $$ by $K(R+\log(k+2))^2$ and thus 
$$\E^{A}\left[\theta(x)\cdot \theta(y) \right] \leq C\log(k+2)\sqrt{ \sum_{u \sim u'} |S(u)-S(u')|^2} \leq C\frac{R+\log(k+2)}{\sqrt{\ln \|x-y\|}}$$	
where the final sums are over all of $\Lambda_n$, we made use of \eqref{eq:Sbound}. As $|\theta(x)\cdot\theta(y)|\leq 1$ we obtain the result.
\end{proof}

\subsection{The harmonic function of the exit set does not fluctuate.}\label{subseq:fnl}
In this section we will prove Proposition \ref{p.small_harmonic_function_for_explorations}: we show that the harmonic function associated to the exit set $A_{R,k}$ does not fluctuate too much - more precisely, we show that $\E\left[\|\phi_{A^i}\|^2(0) \right]=o(1)$ as $n\to \infty$. 

 The main idea of the proof is to use Proposition \ref{pr.MW} to see that even though the absolute values of the GFF near the exploration set are close to $R$, the angles are mixing well enough so that the harmonic extension of the values follows the law of large numbers.  

\begin{proof}[Proof of Proposition \ref{p.small_harmonic_function_for_explorations}]
For simplicity of notation denote $A = \hat A_{R,k}$. 
Recalling the notation $\theta= \frac{\phi}{\|\phi\|}$ and $r=\|\phi\|$, we can write $\E\left[\|\phi_{A}\|^2(0) \right]$ as

	$$\E\left( \sum_{x,y\in A} \nu_x \nu_y \E\left[r_xr_y \E^{r_A}\left[\theta(x)\cdot \theta(y)\right] \mid A\right]\right),  $$
 where $\nu_x, \nu_y$ denote the harmonic measure of $x,y \in A$ seen from $0$ in the connected component of $x \in \Lambda_n \setminus A$.
 By Hölder inequality we can bound this further by
 	$$\E\left( \sum_{x,y\in A} \nu_x \nu_y (\E\left[r_x^4\mid A\right])^{1/4}(\E\left[r_y^4\mid A\right])^{1/4}(\E\left[(\E^{r_A}\left[\theta(x)\cdot \theta(y)\right])^2\mid A\right])^{1/2}\right]. $$
 Each of the two first conditional expectations can be upper bounded using \eqref{eq:2nd} by $C(R+\log(k+2))^4$. The third expectation can be bounded by $1$ when $x = y$. When $x \neq y$ we can use \eqref{eq:condbnd} and bound it by a constant times
$$ \E\left[\E^{r_A}\left[\sum_{(s,t) \in A \times A}d_s p^1_{s\to t} \|S(s)-S(t)\|^2	r_sr_t\right]\mid A\right],$$
which like in the proof of the Mermin-Wagner theorem (there is just no square root) is bounded by $(C\frac{(R+\log(k+2))^2}{\log\|x-y\|})\wedge 1$ for $x \neq y$. Putting everything together we arrive at
$$\E\left[\|\phi_{A}\|^2(0) \right] \leq C(R+\log(k+2))^4\E\left( \sum_{x,y\in A} \nu_x \nu_y \left\{(C\frac{(R+\log(k+2))^2}{\log\|x-y\|})\wedge 1\right\}\right).$$
Now, as $A \subseteq \Lambda_n \setminus \Lambda_{n/3}$, there is a universal constant $c > 0$ such that $\nu_x \leq n^{-c}$. \footnote{ Beurling's estimate would give us a quantitative exponent $c$. We shall only need here the existence of a positive $c>0$ which follows easily by considering concentric dyadic annuli around each given $x$.} 
Decomposing the sum inside the expectation over near-diagonal and off-diagonal parts and using the trivial bound $1$ for the near-diagonal parts, we can bound the expectation by 
$$\E\left(\sum_{x \in A} \nu_x \left(\sum_{\|x-y\| \leq n^{c/3}} n^{-c} + \sum_{\|x-y\| \geq n^{c/3}} \nu_y \left\{(C\frac{(R+\log(k+2))^2}{\log\|x-y\|})\wedge 1\right\}\right)\right) .$$ 
The first sum contains at most $O(n^{2c/3})$ terms on a two-dimensional lattice and thus is bounded by $n^{-c/3}$. The second term we can bound by $$\sup_{\|x-y\| \geq n^{c/3}} \left\{\left (C\frac{(R+\log(k+2))^2}{\log\|x-y\|}\right )\wedge 1\right\} = O\left (\frac{R+\log(k+2)^2}{\log n}\right ),$$
as long as $(R+\log(k+2))^2 = o(\log n)$. We conclude that
$$\E\left[\|\phi_{A}\|^2(0) \right] \leq C\frac{(R+\log(k+2))^6}{\log n}.$$
\end{proof}

\section{Level set percolation of the 2D-GFF}\label{s.perco}
In this section, we will prove Theorem \ref{th.GFFmGFF}. We will first consider the $N$-vectorial GFF rooted at the origin, and then deduce from this the case of the massive GFF. The two propositions thus proved in the two following subsections correspond to the two cases of Theorem \ref{th.GFFmGFF}.

Recall the definition of the event $\{x \overset{\| \cdot \| \leq R, k}\longleftrightarrow  y \}$ defined for the $N$-vectorial GFF $\phi$ : this event holds if and only one can find a sequence of vertices $x_0 = x, x_1, \dots, x_n = y$ such that $d(x_i, x_{i-1}) \leq k$ and $\|\phi(x_i)\| \leq R$ for all $i \in \{0, 1, \dots, n\}$.

\subsection{Level-set percolation of the two-dimensional GFF rooted at $0$.}
The goal of this subsection is to prove that if $\phi:\Z^2\mapsto \R^2$ is a 2D-GFF in $\Z^2$ with value $0$ at $0$, then the sets $\|\phi\| \leq R$ are exponentially clustering. More precisely, 
\begin{proposition}\label{pr.percolation_GFF}
Let $N\geq 2$ and consider the $N$-vectorial GFF on $\Z^2$ which is rooted at the origin. 
Then, for any $R,k>0$, there exists $\psi=\psi(R,k)>0$, such that for any $x,y\in \Z^2$
\begin{align*}\label{}
\Pb{x \overset{\| \cdot \| \leq R,k}\longleftrightarrow  y } \leq e^{-\psi(R,k) \|x-y\|_2}\,.
\end{align*}
\end{proposition}

This is relatively direct given that two-valued sets are not macroscopic, i.e. Theorem \ref{th.ExitSet2}. Indeed, we can tesselate the space with translated copies of $\Lambda_n$, and then see that the probability that $\Lambda_n$ is connected to a $\Lambda_{2n}$ using a path of $L_a^i$ is upper bounded by the corresponding probability for a GFF in $\Lambda_{2n}$ with $0$-boundary condition. By Theorem \ref{th.ExitSet2} we can make this probability arbitrarily small by taking $n$ large, which allows to conclude via standard percolation theory. Let us spell it out.
\begin{proof}

Recall from Theorem \ref{th.isom1} that the law of $\|\phi\|^2$ is equal to that of $L$, the occupation time of a random walk loop soup of parameter $1/2$ in $\Z^2$ and killed at $\{0\}$. From now on, we work only with the random walk loop soup.
	
Tesselate $\Z^2$ is translations of $\Lambda_n$. We call this tesselation $T$ and denote its elements $\square_i$. Furthermore, we denote by $\An_i$ the translation of the annulus $\An_{n,3n/2}$, chosen such that its inner square coincides with $\square_i$. Note that $T$ naturally comes with a graph structure: $\square_i$ neighbours $\square_j$ if one of its sides intersects. With this structure $T$ is graph-isomorphic to $\Z^2$. 
	
Further, define $L_i$ as the occupation time of all loops in the loop soup that do not exit $\An_i \cup \square_i$. Note that $L_i\leq L$ for all $i$ such that $\An_i\cup\square_i$ does not contain $0$. By the isomorphism theorem, $L_i$ is equal in law to $\|\phi_i\|^2$ where $\phi_i$ is a GFF in $\An_i\cup\square_i$ with boundary condition $0$ in the outer boundary of $\An_i$. Also, observe that $L_i$ is independent of $L_j$ as long as $\An_i$ does not intersect $\An_j$, in other words, as long as the distance between $\square_i$ and $\square_j$ is greater than 2.
	
We now define a percolation on $T$ by saying that a square $\square_i$ is open if there is a path $\eta$ going from one boundary to the other in $\An_i$ where $L_i$ is smaller than or equal to $R$, where each step is taken to distance at most $k$ from the previous location. We call such a path a $k-$path. This generates a $2$ dependent percolation whose open clusters contain those of $\{\|\phi\|^2 \leq R\}$, and if we ignore the finite amount of $\square_i$ where $\An_i\cup \square_i$ contains $0$ this percolation is translationally invariant. 

From Theorem \ref{th.ExitSet2}, we directly obtain the following claim \footnote{We include the possible dependency of $k$ on $n$ for later use. It is not used here.}.
\begin{claim}\label{cl.openprob}
For every $C,\eps > 0$, we can find $n$ large enough so that for every $k \leq C\log n$, $$\P(\square_i \text{ is open}) = \P(\text{there is a }k\text{-path crossing }\An_i\text{ on which }L_i \leq R) < \eps.$$
\end{claim}
We can now conclude, as it is well known that two-dependent translation invariant percolation models exhibit exponential decay of cluster size for sufficiently small opening probability. As the existence any cluster of diameter $l$ in $\{\|\phi\| \leq R\}$ would imply the existence of a cluster of diameter $0.2l/n$ in $T$, and this has exponential decay, we conclude.
\end{proof}

\subsection{Level-set percolation of the two-dimensional massive GFF}
We now explain how to obtain a similar percolation result for a massive GFF. \begin{proposition}\label{pr.percolation_massive_GFF}
	Let $N \geq 2$ and consider a massive $N$-vectorial GFF $\phi^{(m)}$ on $\Z^2$. We have that for any $R>0, k\geq 1$, there exists a sufficiently small mass $\bar m=\bar m(R,k)>0$ and a positive $\psi(R,k)>0$ such that for any $m\leq \bar m$ and any  $x,y\in \Z^2$,
\begin{align*}\label{}
\FK{\phi^{(m)}}{}{x \overset{\| \cdot \| \leq R,k}\longleftrightarrow  y  \text{ in } \Z^2} \leq e^{-\psi(R,k) \|x-y\|_2}\,.
\end{align*}
\end{proposition}
\begin{proof}
	The proof is the same as that of Proposition \ref{pr.percolation_GFF}. We only need to explain why we can prove the equivalent of Claim \ref{cl.openprob}.
\begin{claim}\label{cl.massive}
For every $\eps > 0$, we can find $n$ sufficiently large and $m$ sufficiently small such that $k \leq C\log n$, $$\P(\square_i \text{ is open}) = \P(\text{there is a }k\text{-path crossing }\An_i\text{ on which }L_i^m \leq R) < \eps.$$
\end{claim}
To prove this claim, first using Claim \ref{cl.openprob} we choose $n$ so that this probability is less than $\eps/2$ for the non-massive loop soup $L^i$ with zero boundary conditions. Second, via Lemma \ref{l.coupling}, we choose $m > 0$ small enough so that $L_i = L_i^m$ with probability more than $1-\eps/2$ in $\An_i$.

We can now argue exactly in the same way as Proposition \ref{pr.percolation_GFF}.
\end{proof}

\section{The local convergence of $O(N)$ model towards the vectorial GFF as $\beta \to \infty$}\label{s.Taylor}

The following theorem shows that around any fixed point, when we rotate the $O(N)$ model so that it north-pointing there, the $O(N)$ model converges to the $(N-1)$ vectorial GFF rooted to $0$ at this fixed point. 

This theorem is both intuitive and may be considered folklore in the physics literature (appearing e.g. in \cite{polyakov1975interaction}, but also already in \cite{BerBlank}), however, we could not find any proof in the literature and at least the proof presented here requires several non-trivial ingredients. As mentioned earlier, we will need for example to rule out possible fluctuations coming from random harmonic functions in $\Z^2 \setminus \{0\}$.

To simplify our presentation we work in the case of $O(3)$-model and in $\Z^d$ with $d = 2$, but the proof easily extends to any $N\geq 2$ and any $d \geq 2$. 
\begin{theorem}\label{t.convergence to GFF}
	Let $\theta = (\theta^1, \theta^2, \theta^3)$ denote the $O(3)$ model on the torus $\T_{{n}}$ at inverse temperature $\beta$  rooted to point north at $0$, i.e. such that $\theta^3(0) = 1$.
	
	For any sequence of $n=n_\beta$ with $n_\beta \to \infty$ as $\beta \to \infty$, the rescaled angle vector $\sqrt{\beta}(\theta^1, \theta^{2})$ converges in law as $\beta\to \infty$  to the $(N-1)$-vectorial GFF rooted to be $0$ at $0$. Here, the topology corresponds to convergence in law and on compacts subsets of $\T_{n_\beta}$.
\end{theorem}

\begin{remark}\label{r.quant}
From the point of view of Polyakov's approach in \cite{polyakov1975interaction}, it would be interesting to provide a more quantitative version of this convergence. Namely until which quantitative scale $L(\beta)$ do we still see GFF fluctuations. We shall not pursue this here.
\end{remark}
The proof follows in three steps: 
\begin{itemize}
\item we first show tightness via a chessboard estimate, 
\item then argue by a rerooting argument that the limit is equal to GFF plus an independent random harmonic function
\item and finally, via a symmetrisation trick we rule out a possible non-zero harmonic function.
\end{itemize}

\subsection*{Step 1: tightness.}	
	Tightness stems from the following chessboard estimate on individual gradients via an union bound.
	
	\begin{lemma}\label{lemma:cb}
	There exist a constant  $K_0\geq 1$ 
such that for any $n\geq 1$ and $\beta\geq 1$, the following holds.  Consider the $O(3)$ model at inverse temperature $\beta$ in $\T_{n}$ conditioned to have value $(0,0,1)$ at $0$. Then, for all $K\geq K_0$ and for any fixed edge $e$ of $\T_n$, we have 
		$$\P(\|\nabla \theta(e)\| \in [K/\sqrt{\beta}, 2K/\sqrt{\beta}]) \leq \exp(- K^2).$$  

	\end{lemma}
	
	This lemma improves on Proposition C.3 from \cite{GS2}, whose proof would only allow to bound gradients on the scale $\log \beta$. 
	\begin{proof}[Proof of lemma.] Let us note that the probability of the event in question remains the same if one removes the constraint of having value $(0,0,1)$ at $0$. We shall work in this case. 		
	
	Denote by $E_K$ the event that $ \|\nabla \theta(e)\| \in [K/\sqrt{\beta}, 2K/\sqrt{\beta}]$ for all horizontal edges $e$. From the chessboard estimate (\cite{FrohlichLieb1978}) for the $O(3)$ model we have that
	$$\P\left (\|\nabla \theta(e)\| \in [K/\sqrt{\beta}, 2K/\sqrt{\beta}]\right ) \leq \P(E_K)^{2/n^2}.$$
	We can write 
	$$\P(E_K) = \frac{1}{Z_\beta}\int_{(\S^2)^{\T_n}} 1_{E_K} \exp\left (-\frac{\beta}{2} \sum_e \|\nabla \theta(e)\|^2_2\right )\prod_{x\in \T_n}d\theta(x).$$
	Now, notice that on the event $E_K$, we have
	 $$\frac{\beta}{2} \sum_e \|\nabla \theta(e)\|^2_2 \geq n^2K^2.$$
	We further claim that there is some $C >0$ such that 
	$$\int_{(\S^2)^{\T_n}} 1_{E_K}\prod_{x\in \T_n}d \theta(x)\leq \left(C\frac{K^2}{\beta}\right)^{n^2}$$
	Indeed, this can easily be done by integrating out vertex by vertex and see that the integration range of values it can take is at most of the order $\frac{K^2}{\beta}$ for each vertex other than the very first vertex chosen.
	
On the other hand, we can bound 
\begin{align*}
	Z_\beta \geq \int_{(\S^2)^{\T_n}} 1_{F_{K}} \exp\left (-\frac{\beta}{2} \sum_e \|\nabla \theta(e)\|^2_2\right )\prod_{x\in \T_n}d \theta(x),
\end{align*} where $F_{K}$ is the event that all $\theta(x)$ are in the $K/2\sqrt{\beta}$ neighbourhood of the north pole. On the event $F_K$ we have that 
		\begin{align*}
			\frac{\beta}{2} \sum_e \|\nabla \theta(e)\|^2_2 \leq n^2K^2/4\,.
		\end{align*}
Similarly as above, we also claim that there exists another universal constant $c>0$ such that 
$\int_{(\S^2)^{n^2}} 1_{F_K} \geq \left(c\frac{K^2}{\beta}\right)^{n^2}$. Hence we obtain,
	\begin{align*}
		\P(\|\nabla \theta(e)\| \in [K/\sqrt{\beta}, 2K/\sqrt{\beta}] \leq  \left (\frac C c\right )^2 \exp\left (-\frac 3 2 K^2\right ) \leq \exp(- K^2)
	\end{align*} by choosing $K$ large enough and conclude.
	\end{proof}
	
	Using this lemma we can take $K = 2^m, 2^{m+1},\dots$ and obtain from the union bound that for all $K$ large enough 
	\begin{equation}\label{e.exponentialtail}\P(\|\nabla \theta(e)\| \geq K/\sqrt{\beta}) \leq e^{-\frac 1 2 K^2}.
	\end{equation} In particular, we obtain tightness on compacts in the sup norm. This is summarised in the following lemma.
	\begin{lemma}\label{l.tightness}Take any sequence $\beta \to \infty$ and $n_\beta \to \infty$ and let $\theta_\beta$ be an  $O(3)$ model on the torus $\T_{n_\beta}$ at inverse temperature $\beta$ and that takes values $(0,0,1)$ at the point $0$. We have that the sequence $\sqrt \beta(\theta^1_{\beta},\theta_\beta^2)$ is tight for the topology of uniform convergence on compact sets of $\Z^2$.
		
	Furthermore, choose a subsequence of $\beta$ (that we denote the same way) such that the sequence $\sqrt{\beta}(\theta^1, \theta^{2})$ converges on compacts to some limit $\tilde \phi$. For any edges $e,e' \in E(\Z^2)$, we have that
	\begin{align}\label{e.2-points grandient}
		\E_\beta\left[\nabla \theta_\beta^{1,2}(e) \nabla \theta_\beta^{1,2} \right] \to \E\left[\nabla \tilde \phi (e) \nabla \tilde \phi(e') \right] 
	\end{align}
	\end{lemma}
\begin{proof}
	The first part follows directly \eqref{e.exponentialtail}. The second part follows also from \eqref{e.exponentialtail}, as it implies all the moments of $\nabla \theta^{1,2}_\beta(e)$ are bounded.
\end{proof}

	\subsection*{Step 2: Any subsequential limit $\tilde \phi$ is equal to GFF plus an harmonic function.}
	We need three ingredients to conclude this step. First, we use a direct calculation to obtain a certain format for the limiting density on finite boxes.
	
	\begin{lemma}
	Let $\tilde \phi$ be any subsequential limit of $\sqrt{\beta}(\theta^1, \theta^{2})$. Then for any finite box $\Lambda \subset \Z^2$ containing the origin, the law of $\tilde \phi|_{\Lambda}$ can be written as
$$\frac{1}{Z} \exp\left (-\frac{1}{2} \sum_{e \in E(\Lambda)}\|\nabla \tilde \phi(e)\|^2\right )  \delta_{\tilde \phi(0)} \Pi_{x \in \Lambda \backslash \partial \Lambda} \, d \tilde \phi(x) \,  \d\nu (\tilde \phi_{|\p \Lambda})\,,$$
with $d$ denoting the Lebesgue measure and $d \nu$ some to be determined measure on $\R^{\p \Lambda}$. 
	\end{lemma}
	
	\begin{proof}
	Let us note that the law of $(\theta\md_{\Lambda})$ given $\theta \mid  _{\partial\Lambda}=\bar \theta$ is proportional to
	\begin{align}\label{e.density}
		\exp\left (\beta \sum_{i\sim j} \theta^{1,2}(i)\cdot \theta^{1,2}(j) + \theta^3(i) \theta^3(j) \right ) \prod_{x\in \Lambda \backslash  (\partial \Lambda \cup \{0\})}d^{\S^2} \theta(x),
	\end{align}
where $\theta^{1,2}=(\theta^1,\theta^2)$. Furthermore, let us condition on the event
\begin{align*}
	\K_\beta:=\left \{ \|\nabla \theta(e)\| \leq \frac K {\sqrt{\beta}}: \text{ for all } e\in E(\Lambda)\right \}.
\end{align*}
By Lemma \ref{lemma:cb}, we have that there exists $K=K(\Lambda)\geq 1$ large enough such that for any $\delta>0$ and uniformly for all $\beta\geq 1$, $\P_\beta(\K_\beta)\geq 1-\delta$.

Now, consider $\beta$ large enough so that $ \diam(\Lambda) \,  \frac K {\sqrt\beta} <1$. Then on the event $\K_\beta$, we have that $\theta^3(x)>0$ for any $x \in \Lambda$. As a consequence, on the event $\K_\beta$ we can determine $\theta^3\mid_{\Lambda}$ from the value of $\theta^{1,2}\mid_{\Lambda}$. Thus, conditionally on $\K_\beta$ and $\theta^{1,2}\mid_{\partial \Lambda}=\bar \theta^{1,2}$, the law of $\phi_\beta^{1,2} := \sqrt{\beta}\theta^{1,2}_{\md \Lambda}$ is proportional to
\begin{align}\label{e.density2}
	&\exp\left (\sum_{i\sim j} \phi_\beta^{1,2}(i)\cdot \phi_\beta^{1,2}(j) + \beta\left (\sqrt{1-\frac{\|\phi_\beta^{1,2}(i)\|^2}{\beta}}\sqrt{1-\frac{\|\phi_\beta^{1,2}(j)\|^2}{\beta}}-1 \right ) \right )\\
\nonumber	&\hspace{0.6\textwidth}\1_{\K_\beta}\prod_{x\in \Lambda \backslash (\partial \Lambda\cup \{0\})}d^{\sqrt{\beta} \S^2}\phi_\beta(x),
\end{align}
Here $\phi_\beta^{1,2}(i)= \sqrt{\beta}\, \bar  \theta^{1,2}(i)$ when $i\in \p \Lambda$ and $\phi_\beta^{1,2}(0):=0$. As we can write $$d^{\sqrt{\beta} \S^2}\phi_\beta(x) = \frac{c_{\mathbb{S}^2}}{\sqrt{1 - \frac 1 \beta \|\phi_\beta^{1,2}(x)\|^2}} 1_{\phi_\beta^{1,2}(x) \in \sqrt{\beta}\D} d\phi^1_\beta(x)d\phi^2_\beta(x),$$ we obtain as $\beta\to\infty$ convergence to a measure which is proportional to
 
\begin{align*}\label{e.dens 3}
	&\exp\left (\sum_{i\sim j} \tilde \phi(i)\cdot \tilde \phi(j) - \frac{1}{2}\|\tilde \phi (i)\|^2-\frac{1}{2}\|\tilde \phi (j)\|^2 \right ) \1_{\K}\prod_{x\in \Lambda\backslash  (\partial \Lambda\cup \{0\})}d\tilde \phi(x),\\
	&= \exp\left (-\frac{1}{2}\sum_{i\sim j}(\tilde \phi(i)-\tilde \phi(j))^2 \right )\1_{\K}\prod_{x\in \Lambda\backslash  (\partial \Lambda\cup \{0\})}d\tilde \phi(x).
\end{align*}
where 
\begin{align*}
\K:=\{ \|\nabla \tilde \phi(e)\| \leq  K: \text{ for all } e\in E(\Lambda)\}.
\end{align*}
Finally, from Lemma \ref{lemma:cb} it also follows that with arbitrary high probability over $\tilde \phi \mid_{\partial \Lambda }$ the conditional probability of $\K_\beta$ is still arbitrary close to $1$ as $K \nearrow \infty$; thus we conclude.
	\end{proof}
	
	The orthogonal decomposition of the Dirichlet energy thus tells us that any subsequential limit $\tilde \phi$ satisfies the domain Markov property of the Gaussian free field on finite boxes $\Lambda$. Using this, we can identify $\tilde \phi$ on $\Z^2$ as a sum of a random function that is $0$ at the origin and harmonic elsewhere, and an independent GFF rooted to be $0$ at the origin.	
	\begin{lemma} \label{l.GFF+harmonic}
		Let $\tilde \phi$ be a random function $\Z^2 \to \R^{2}$ such that $\tilde \phi(0) = 0$ and $\tilde \phi$ satisfies the domain Markov property of the Gaussian free field on finite boxes $\Lambda$ in the following sense: 
		\begin{itemize} 
		\item for any finite box $0\in \Lambda \subset \Z^2$, $\tilde \phi$ can be decomposed on $\Lambda$ into an independent sum of a zero boundary GFF rooted to $0$ at $0$ and a random harmonic function, whose value is zero at $0$. 
		\end{itemize}
		Then $\tilde \phi$ can be written as $\phi + h$, where $\phi$ is a $2$-vectorial GFF rooted to be $0$ at $0$ and $h$ is a (random) function that is harmonic everywhere except at the origin, is equal to $0$ at $0$ and is independent of $\phi$.
	\end{lemma}
	
	\begin{proof}
	By the condition of the lemma, we can write in each box $\Lambda_n$ the field $\tilde\phi$ as the sum of a zero boundary GFF $\phi_n$ inside $\Lambda_n$ and an independent random harmonic function $h_n$.  We now argue in two steps. 
	
	First, it is known that the sequence of zero boundary GFF on $\Lambda_n$ rooted to be $0$ at $0$ converges in law to the whole plane GFF rooted to be $0$ at $0$. 
	
	Second, we argue that the harmonic part $\{h_n\}_n$ is tight as $n\to \infty$. To see this, notice that in law, for any $n\geq 1$, $h_n = \tilde \phi_{\md \Lambda_n} - \phi_n$. Since both $\tilde \phi_{\md \Lambda_n}$ and $\phi_n$ are tight and even converge as $n\to \infty$,   we obtain that the random harmonic function $h_n$ converges in law as $n\to \infty$ as well. 
	\end{proof}
	
	Finally, to obtain harmonicity of $h$ at the origin, we need the following rerooting lemma.
	
	\begin{lemma}\label{l.rerooting}
		Let $v \in \Z^2$. Then any subsequential limit $\tilde \phi_{v}$ of $\sqrt{\beta}(\theta^1, \dots, \theta^{N-1})$ rooted to be $0$ at $v$ satisfies the following rerooting property. We have that $(\tilde \phi_{v}(z))_{z \in \Z^2}$ is equal in law to $(\tilde \phi(z) - \tilde \phi(v))_{z \in \Z^2}$.	
	\end{lemma}
This implies that the function $h$ from Lemma \ref{l.GFF+harmonic} is in fact harmonic everywhere on $\Z^2$. Indeed, as $\tilde \phi$ is equal to the independent sum of $\phi$ and $h$ with $h$ harmonic everywhere but $0$, we obtain from this lemma that $\Delta h(0)$ is equal in law to $\Delta h(v)$. But the latter is a.s. equal to $0$, and thus is the former.	
	\begin{proof}[Proof of Lemma \ref{l.rerooting}.]
	The original $O(3)$ model on $\T_n$ has the following rerooting invariance: we pick any point $v \in \T_n$ and then do a global rotation $R_v$ to bring $\theta(v)$ to the north pole, then $(R_v\theta(z+v))_{z \in \T_n} $ is equal in law to $(\theta(z))_{z \in \T_n}$. 
	Moreover, this rotation matrix can be explicitly written down using the Rodrigues rotation formula: $R_v = I + Q_v + Q_v^2\frac{1}{1+\cos \alpha}$ where $\alpha$ is the angle between the north pole and
	\begin{align*}
		Q_v=\begin{pmatrix}
			0 & 0 & -\theta^2(v)\\
			0 & 0 & -\theta^1(v)\\
			\theta^2(v) & \theta^1(v)&0
		\end{pmatrix}.
	\end{align*}
	 Applying this to $\sqrt{\beta}\theta(z)$ we obtain that $(R_v \sqrt{\beta} \theta^i(z)) = \sqrt{\beta}(\theta^i(z) - \theta^i(v)) + o(\beta (\theta^{i}(v)^2 + \theta^{1}(v)\theta^{2}(v))$ for $i = 1,2$. Using now Lemma \ref{lemma:cb} to see that the error term goes to zero along any subsequence uniformly in $z$, we obtain the lemma.
	\end{proof}
	In the reminder of the proof we will study this harmonic function $h$ more closely and show it must be in fact equal to zero.  
	
	\subsection*{Step 3: harmonic function is equal to zero.}	We argue in two steps. First, we start with the following classical lemma (see for example \cite{heilbronn1949discrete}) which implies that $h$ needs to be linear.
	\begin{lemma}
		Let $\tilde h$ be a (multidimensional) harmonic function on $\Z^2$ such that $\|\tilde h(v)\| = O(\log d)$ for all vertices of distance $d$. Then $\tilde h$ is constant.
	\end{lemma}
	As the proof is short, we provide it for the convenience of the reader.
	\begin{proof}[Proof of lemma.]
		Let $B(0,r)$ be some fixed ball around the origin and consider some bigger ball $B(0,R)$. Then by the discrete Beurling theorem, we see that uniformly in $B(0,r)$, we have that each $\tilde h(v)$ and $\tilde h(w)$ agree with probability larger than $1-(r/R)^\alpha$ for some $\alpha > 0$. Moreover, on the opposite event, by using the Poisson representation of the harmonic function $\tilde h$ and the growth condition, we see that $\|\tilde h(w) - \tilde h(v)\| = o(R).$ Thus letting $R \to \infty$ we obtain that $\tilde h(r)$ is constant inside $B(0,r)$ and hence in fact on whole of $\Z^2$.
	\end{proof}

	\begin{corollary} \label{c.harmonic->linear}Let $\tilde \phi$ be an accumulation point of $\sqrt{\beta}(\theta^1, \theta^{2})$. Write $\tilde \phi=\phi + h$, where $\phi$ is a GFF in $\Z^2$ taking value $0$ at $0$ and $h$ be a harmonic function. We have that $h$ is linear.
	\end{corollary}
	\begin{proof}
			Consider $\tilde h(v) := h(v) - h(v+e)$, where $e$ is the unit vector in the direction of the $x-$axis. By taking $K = C\log d$ in Lemma \ref{lemma:cb}, we see from the union bound and Borel-Cantelli that almost surely, for all $d\geq 2$ and all vertices $v$ at distance $d$ from the origin, $\| {\nabla} \tilde \phi(v)\| = O(\log d)$. 
			
Using the fact that gradients of the 2D GFF also grow at most like $O(\log d)$ almost surely, we obtain that $\|\tilde h(v)\| = O(\log d)$ for all all vertices of distance at most $d$ almost surely. Thus,  we can conclude from the above lemma that $\tilde h$ is almost surely constant. Repeating the same argument with $e$ being the unit vector in the direction of the $y-$axis we deduce that $h$ must be linear.
	\end{proof}
	To finish off, we need to argue that the slope of the linear function $h$ is zero. The following claim shows that a certain correlation of gradients is always non-positive for the $O(N)$ model while it is non-positive and tending to zero with distance for the vectorial GFF. 
	
	\begin{lemma}\label{l.2-point gradient negative}
		Let $\phi$ be a vectorial GFF and $e_1$, $e_2$ be two edges along the real axis or the imaginary axis of even distance. Then $\E(\nabla \phi(e_1) \nabla \phi(e_2)) \leq 0$ and moreover $\E(\nabla \phi^i(e_1) \nabla \phi^i(e_2)) \to 0$ as the distance between $e_1$ and $e_2$ goes to infinity for any $i$-th coordinate of $\phi$. 
		
		Similarly, let $\theta$ bet the $O(N)$ model on $\T_{2n}$. Then if $e_1$, $e_2$ are any two edges along one line of the torus of even distance, we have that  $\E(\nabla \theta^i(e_1) \nabla \theta^i(e_2)) \leq 0$ for any $i$-th component of $\theta$.
	\end{lemma}
	
	\begin{proof}
		All properties follow from a similar argument, which we will give in the case of the $O(N)$ model: we condition on the values of the $O(N)$ models on the vertices $V_S$ of $\T_{2n}$ on the two lines perpendicular to the line of $e_1, e_2$ that separate them symmetrically. Then the two conditional laws on the two connected components of $\T_n \backslash V_S$ are equal and independent. Moreover under this law, by symmetry $\nabla \theta^i(e_1) \sim -\nabla \theta^i(e_2)$. Thus we deduce that 
		\begin{align*}
			\E(\nabla \theta^i(e_1) \nabla \theta^i(e_2)) = \E\left(\E(\nabla \theta^i(e_1)| \theta_{V_S}) \E(\nabla \theta^i(e_2)| \theta_{V_S})\right) = - \E\left(\E(\nabla \theta^i(e_1)| \theta_{V_S})^2\right),
		\end{align*}
	which is clearly negative.
	\end{proof}

\subsection*{Step 4: conclusion.}
We now collect the ingredients to prove Theorem \ref{t.convergence to GFF}.
\begin{proof}[Proof of Theorem \ref{t.convergence to GFF}]
	Take any sequence $\beta\to \infty$ and $n_\beta\to \infty$. We know by Lemma \ref{l.tightness} that the sequence $\sqrt \beta \theta^{1,2}_\beta$ is tight. We will conclude by showing that all accumulation points $\tilde \phi$ have the law of a GFF in $\Z^2$ that takes value $0$ at $0$. By Lemma \ref{l.GFF+harmonic} and \ref{l.rerooting}, we know that the law of $\tilde \phi$ can be written as $\phi + h$, where $\phi$ has the law of the desired GFF and $h$ is an independent random harmonic function. Further, using Corollary \ref{c.harmonic->linear} we see that $h$ is in fact linear.
	
	To finish off, we need to argue that the slope of the linear function $h$ is zero. To do this notice the following: let $e_1$, $e_2$ be two edges along the real axis and $f_1$ and $f_2$ two vertices along the imaginary axis. If the slope of the random linear function $h$ is non-zero with positive probability, then either $\E(\nabla h^i(e_1) \nabla h^i(e_2)) > 0$ or  $\E(\nabla h^i(f_1)\nabla h^i(f_2)) > 0$ uniformly over the distance for any $i$-th coordinate of the function $h$. However, this is not possible thanks to Lemma \ref{l.2-point gradient negative}. From this, we conclude.
\end{proof}

\section{Revisiting the approach of A. Patrascioiu and E. Seiler on disproving exponential decay and geometric interpretation of the two-point correlation function in the spin $O(N)$ model}\label{s.PS}

By perturbation theory and renormalization flow arguments \cite{polyakov1975interaction}, it is conjectured that 2D non-abelian continuous spin models like $O(N)$ for $N \geq 3$ and 4D non-abelian Yang-Mills theories exhibit a mass gap, i.e. exponential decay of correlation at all temperatures. 

For the time being, no mathematical proof of such a statement exists and there are also arguments from the physics community in the opposite direction. Possibly most prominently A. Patrascioiu and E. Seiler have been developing heuristic arguments against the existence of a mass gap. In this section we will revisit their central percolation-based argument, as presented e.g. in \cite{patrascioiu1992phase,patrascioiu1993percolation,patrascioiu2002percolation} (see also the review paper \cite{seiler2003case}) for the Heisenberg model. 

We will reformulate their argument in a setting of metric graph Heisenberg model, and we shall make part of their argument rigorous. However, we will also give a counterexample to another part of their argument and conjecture that their main assumption may in fact not hold. Let us recap their argument to be able to state things more clearly:

\begin{enumerate}
	\item First, the authors claim that it suffices to prove this result on a restricted model: for some $1/2 \gg \delta > 0$, one restricts the $O(3)$ model to satisfy $\theta(x)\cdot \theta(y) \in [1-\delta, 1]$ for any neighbouring vertices $x, y$. \footnote{This restricted modified  model has been successfully analyzed in \cite{aizenman1994slow}, where it was shown that the BKT phase holds for all sufficient small cut-off $\delta$ (which correspond to ``small'' physical temperature in this setting).}	
	\item Now, we pick some {$\eps > 2 \delta$} and separate all vertices in $\Z^2$ into three parts: $E_\eps := \{z: \theta^3(z) \in [-\eps, \eps]\}$, ${N}_\eps := \{z: \theta^3(z) > \eps\}$ and ${S}_\eps := \{z: \theta^3(z) < -\eps\}$. Notice that by the choice of $\eps$, we see that the connected components $S_\eps$ and $N_\eps$ are disjoint.
	\item The authors then claim: if $E_\eps$ does not percolate, then the Heisenberg model does not have exponential decay. This is proved modulo certain very believable hypothesis on percolation properties, and basically goes as follows: 1) one argues that neither $S_\eps$ nor $N_\eps$ percolates via uniqueness of infinite cluster. 2) One deduces that their expected cluster sizes are infinite via a lemma by Russo 3) one deduces that there is no exponential decay.
	\item The authors then state their belief that $E_\eps$ does not percolate 
	\item {Finally}, for the {sake} of security, they also provide a heuristic argument why even if $E_\eps$ should percolate, there would not be exponential decay at sufficiently low temperatures. We will return to this argument.
\end{enumerate}

In this section we will reformulate this strategy by making use of a certain extension of the $O(N)$ model to the metric graph, that changes continuously over edges of $\Z^2$ but whose restriction to the vertices is the original $O(3)$ model. In this setting there is no need to restrict the angles of the $O(3)$ model {as is done in \cite{patrascioiu1992phase,patrascioiu1993percolation})}, because one can separate different sign clusters of $\theta^3$ simply via its zero set, denoted by $E$ for the equator. In this framework 
\begin{itemize}
\item We give a rigorous proof of point 3) above: we show that if $E$ (or more precisely the set $\hat E$ to be defined below in Section \ref{ss.hatE}) does not percolate for some value of $\beta$, then there is {polynomial decay} of correlation and hence no mass gap (see Corollary \ref{pr.PSB}. This relies on a percolation representation of the correlation functions proved in the next subsection.
\item We also provide a counterexample for the heuristic argument of point 5) in the setting of the original $O(3)$ model (see Theorem \ref{th.example}).
\end{itemize}

In particular, this means that if one believes in the conjecture of Polyakov, one is bound to believe that the equator in the metric graph model, $E$, does percolate for any $\beta > 0$. It would be interesting to study this further and to also obtain rigorous results of the type - if $E$ does percolate, then there is exponential decay.

\subsection{Setup: an extension of the $O(N)$ to the metric graph.}

To state our result, we first have to define our percolation model. It will be defined on the metric graph $\tilde \Lambda_n$ associated to $\Lambda_n$ as follows. This extension to the metric graph is inspired by the work \cite{lupu2016loop} and has been used recently in the context of $\Z$ and compact valued spin systems in the works \cite{van2021elementary,aizenman2021depinning,dubedat2022random}. We also refer to these works for the definition of such extensions to the cable graph.

\begin{definition}[A simple extension of $O(N)$ model to the metric graph]\label{def.metricON}
Consider $\Phi=(\tilde \Phi^1,...,\tilde \Phi^{{N}})$ an $N$-vectorial GFF with free boundary in the metric graph $\tilde \Lambda_n$. Define $\tilde \theta (x) := \tilde \Phi(x) /\|\tilde \Phi(x)\|$ for any $x \in \tilde \Lambda_n$.

Then the model, conditioned on the fact that $\|\Phi(v)\|=\sqrt{\beta}$ for all vertices $v\in \Lambda_n$, is called the simple extension of $O(N)$ model to the metric graph. It holds that $\tilde \Phi/\|\tilde \Phi\|=\tilde \Phi/\sqrt{\beta}$ restricted to $\Lambda_n$ has the law of an $O(N)$-model at inverse temperature $\beta$ on $\Lambda_n$.
\end{definition}

The final part comes directly from Proposition \ref{pr.anglesGFF}. In this setting, for the Heisenberg model the equator $E$ is just given by the set $\{x \in \tilde \Lambda_n: \tilde \theta^3(x) = 0\}$. Further, correlation functions of $\theta^3$ can be just rephrased using percolation properties of $\theta^3 > 0$, which in turn can be expressed using loop soups, thanks to the isomorphism theorems (see Theorem \ref{th.corr} ). This is collected in the following proposition.

\begin{proposition} 
 Fix $m,\beta\geq 0$ and let $\tilde \L$ denote a loop soup of parameter $N/2$ with mass $m$ in the metric graph $\tilde \Lambda_n$ conditioned on the fact that its occupation time $L$ restricted to the vertices is equal to $\sqrt{\beta}$. Further, for each loop in $\tilde \L$ choose uniformly a label in $\{1,..,{N}\}$ independently of each other and define $\tilde \L^{(i)}$ to be the set for all loops that share the label $i$. Denote by $\tilde L^{(i)}$ the law of the occupation time of $\tilde \L^{(i)}$ and by $L^{(i)}$ its restriction to the vertices of $\Lambda_n$.

Now, define $(s_i)_{i=1}^N$  independent functions of the metric graph taking value in $\pm 1$, whose law is symmetric and satisfies the following:  $s_i$ is constant on every connected component of $\tilde \L^1 > 0$, and it is independent in different connected components. 

We have that $\beta^{-1/2}(s_1 \tilde \L^{(i)}, \dots, s_N \tilde \L^{(N)})$ is equal in law to the extension of the $O(N)$ model on the metric graph $(\tilde \theta^1, \dots, \tilde \theta^N)$ given just above.
\end{proposition}

\subsection{Geometric interpretation of the two-point correlation function in the spin $O(N)$ model.}

The main result of this subsection says that the two-point correlation function can be bounded by the probability that $x,y$ belong to the same connected component of $\tilde \L^1 > 0$. We denote the relevant event by $x\stackrel{\tilde \L^1}{\longleftrightarrow }y$.

\begin{theorem}\label{pr.connectivity O(n)}
	Let $x$ and $y$ be two vertices of $\Lambda_n$ and $\theta$ be an $O(N)$-model in $\Lambda_n$ at inverse temperature $\beta$ with free-boundary condition in $\partial \Lambda_n$. We have that
	\begin{align}
		\frac{1}{N} \P(x\stackrel{\tilde \L^1}{\longleftrightarrow} y ) \leq \E\left[\theta(x)\cdot \theta(y) \right] \leq N\, \P(x\stackrel{\tilde \L^1}{\longleftrightarrow } y).
	\end{align}
\end{theorem}
\begin{remark}
Our result is reminiscent of the work \cite{campbell1998isotropic} which obtained FK-type representations of two-point correlation functions of the spin $O(N)$ model in the case $N=3$. As our proof does not rely on Ginibre's inequality (\cite{Ginibre}), our statement holds for any value of $N$. Note that this is also the case in the work \cite{dubedat2022random} which does not rely on Ginibre either. 
\end{remark}

\subsubsection{Representation using local times and the upper bound.}
The first step to prove Theorem \ref{pr.connectivity O(n)} is the following representation of the two point correlation using the loop soup.
\begin{lemma}
	\label{l.connectivity O(n)} Let $x$ and $y$ be two vertices of $\Lambda_n$ and $\theta$ be an $O({N})$-model in $\Lambda_n$ at inverse temperature $\beta$ with free-boundary condition in $\partial \Lambda_n$. We have that
	\begin{align}\label{eq.2ptls}
		\frac{{N}}{\beta}\E\left[\sqrt{L^1}(x)\sqrt{L^1}(y)\mid x\stackrel{\tilde \L^1}{\longleftrightarrow}y\right] \P(x\stackrel{\tilde \L^1}{\longleftrightarrow} y ) = \E\left[\theta(x)\cdot \theta(y) \right] . 
	\end{align}
\end{lemma}
\begin{proof}
	Consider $\Phi=(\tilde \Phi^1,...,\tilde \Phi^{{N}})$ an $N$-vectorial GFF with free boundary in the metric graph $\tilde \Lambda_n$ conditioned on the fact that $\|\Phi(v)\|=\sqrt{\beta}$ for all vertices $v\in \Lambda_n$. Recall from Definition \ref{def.metricON} that then $\tilde \Phi/\|\tilde \Phi\|=\tilde \Phi/\sqrt{\beta}$ restricted to $\Lambda_n$ has the law of an $O(N)$-model at inverse temperature $\beta$ on $\Lambda_n$.
	
	Now let $A \subseteq \tilde \Lambda_n$ denote the (metric graph) connected component of $x$ in the set $\tilde \Phi^1 > 0$. Note that the law of $\tilde \Phi^1$ restricted to $\tilde \Lambda_n\backslash A$ conditioned on the set $A$ is symmetric, in particular $\E\left[\tilde \Phi^1(z)\mid A \right] =0$ for all $z\in \tilde \Lambda_n\backslash A$. Thus,
	\begin{align*}
		\E\left[ \theta(x)\cdot \theta(y)\right]= \frac{N}{\beta}\E\left[\tilde \Phi^1(x) \tilde \Phi^1(y) \right]= \frac{N}{\beta}\E\left[|\tilde \Phi^1(x)| |\tilde \Phi^1(y)|\1_{  x\stackrel{\tilde \L^1}{\longleftrightarrow } y} \right].
	\end{align*}

\end{proof}
The upper bound of Theorem \ref{pr.connectivity O(n)} now follows:
indeed, using that $|\tilde \Phi^1|^2$ is equal in law to $(L^1)^2$ and that $|\tilde \Phi^1(x)||\tilde \Phi^1(y)|\leq \beta$ we conclude from \eqref{eq.2ptls} that
$$\E\left[\theta(x)\cdot \theta(y) \right] 
\leq N\, \P(x\stackrel{\tilde \L^1}{\longleftrightarrow } y )$$

\subsubsection{Conditional FKG for the $O(N)$-model and the lower bound.}
To prove a lower bound $\E\left[\sqrt{L^1}(x)\sqrt{L^1}(y)\mid x\stackrel{\tilde \L^1}{\longleftrightarrow }y\right]$, we observe that there is a nice conditional FKG for the $O(N)$-model.

\begin{lemma}\label{pr.FKG O(n)}
	Let $f,g$ be increasing functions in $\Lambda_n$. We have that
	\begin{align*}
		\E\left[f(\tilde L^1)g(\tilde L^1)\mid \tilde \L\right] \geq \E\left[f(\tilde L^1)\mid \tilde \L\right] \E\left[ g(\tilde L^1) \mid \tilde \L\right] 
	\end{align*}
\end{lemma}
\begin{proof}
Observe that the law of $\tilde \L^1$ conditioned on $\tilde \L$ is given by sampling a random Bernoulli random variable of parameter $1/N$ for each loop in $\tilde \L$. Thus, $\tilde L^1$ is an increasing function of a Bernoulli percolation and we conclude using FKG inequality for Bernoulli percolation.
\end{proof}

Using this lemma, we can obtain the following bound.
\begin{lemma}\label{l.two-point conditioned}In the context of Theorem \ref{pr.connectivity O(n)}
	\begin{align*}
		\E\left[\sqrt{L^1}(x)\sqrt{L^1}(y)\mid x\stackrel{\tilde \L^1}{\longleftrightarrow}y\right]\geq \frac{\beta}{{N}^2}.
	\end{align*}
\end{lemma}
\begin{proof}
As the event $ x\stackrel{\tilde \L^1}{\longleftrightarrow}y$ is increasing in $\tilde L^1$, we conclude from Lemma  \ref{pr.FKG O(n)} that
	\begin{align*}
		\E\left[\sqrt{L^1(x)}\sqrt{L^1(y)}\mid x\stackrel{\tilde \L^1}{\longleftrightarrow }y, \tilde \L\right]&\geq \E\left[\sqrt{ {L^1(x)}}\sqrt{ {L^1(y)}}\mid  \tilde \L \right] \\
		&\geq \E\left[\sqrt{ {L^1(x)}}\mid \tilde \L\right ]\E\left [\sqrt{ {L^1(y)}}\mid  \tilde \L \right].
	\end{align*}
Further, because $L^1\leq \beta$, we have that
\begin{align*}
	 \E\left[\sqrt{\frac{L^1(x)}{\beta}}\mid \tilde \L\right ]\geq \E\left[\frac{L^1(x)}{\beta} \mid  {\tilde \L} \right].
\end{align*}
Finally, due to the symmetry of the construction
\begin{align*}
	\E\left[\frac{L^1(x)}{\beta} \mid  {\tilde \L} \right]= \E\left[\frac{L^i(x)}{\beta} \mid  { \tilde \L} \right]=\frac{1}{ {N}}.
\end{align*}
We conclude now by averaging over all possible $\tilde \L$.
\end{proof}

The upper bound in Theorem \ref{pr.connectivity O(n)} follows directly from this lemma and \eqref{eq.2ptls}

\subsection{Metric graph $O(N)$ model and percolation of its equator.}\label{ss.hatE}
Given Theorem \ref{pr.connectivity O(n)}, the fact that non-percolation of $E$ prohibits exponential decay is an easy corollary. Recall that for the Heisenberg model, the equator $E$ was the zero set (in the cable graph) of $\theta^3(z)$. To talk about percolation, we define $\hat E$ to be the set of all edges of the dual lattice of $\Z^2$ that cross an edge $e$ of the initial lattice that intersects $E$. We claim the following. 

\begin{corollary}\label{pr.PSB}
	If $\hat E$ does not percolate for some $\beta > 0$, then there is some {$C > 0$} such that for some sequence of $x \neq y \in \Z^2$ with $|x-y| \to \infty$, we have that $\E \theta^3(x)\theta^3(y) > C|x-y|^{-{3}}$. 
\end{corollary}

Although the proof is classical, we still write the short argument here for completeness.

\begin{proof}
{
We shall follow a classical idea which goes back to Proposition 1 in \cite{russo1978note} (see also \cite{van2021elementary,aizenman2021depinning} for a recent use in the context of $XY$ and Villain models).

By construction the random set $\hat E$ (which can be seen as a configuration in $\{0,1\}^{E(\Z^2)^*}$) is the dual of the percolation configuration $\omega\in \{0,1\}^{E(\Z^2)}$ defined by $\omega(e):=1$ if and only if the component $\theta^3(x)$ does not vanish along the cable $x\in e$ (boundary points included). 

If $\hat E$ does not percolate then, it implies that a.s. for any $R\geq 1$, there is a connected cluster which connects the semi-infinite horizontal line $\{(n,0), n\geq R\}$ with the semi-infinite vertical line $\{(0,n), n \geq R\}$. Using Borel-Cantelli Lemma, this readily implies that the series 
\begin{align*}\label{}
\sum_{n\geq 1 ,m\geq 1} \Pb{ (n,0) \overset{\omega
}\longleftrightarrow (0,m)}
\end{align*}
needs to diverge.  This implies that there exists a sequence of pairs of integers $(n_k,m_k)$ such that $n_k+m_k \to \infty$ and 
\begin{align*}\label{}
\Pb{ (n_k,0) \overset{\omega
}\longleftrightarrow (0,m_k)} \geq  \left(\frac 1 {n_k +m_k}\right)^3\,.
\end{align*}
We conclude the proof by noticing that $\Pb{ (n_k,0) \overset{\omega
}\longleftrightarrow (0,m_k)}$ is by definition the probability that there is a connected path $\gamma : (n_k,0) \to (0,m_k)$ along which the component $\theta^3(x)$ does not vanish. Now, by Proposition \ref{pr.connectivity O(n)}, this is the same, up to constraints, as $\EFK{\beta}{}{\theta_{(n_k,0)} \cdot \theta_{(0,m_k)}}$. }
\end{proof}

\subsection{Strongly percolating random environments for the $XY-$model which exhibit exponential decay at arbitrarily low temperatures.}

In this section we provide the counterexample to point 5) of the above sketch and show that even though the $XY$ model undergoes the $BKT$ transition, one can still for every inverse temperature construct {ergodic} and strongly percolating subgraphs of $\Z^2$, converging to $\Z^2$ such that $XY$ model on those environments exhibits exponential decay. We believe that this counterexample could be of independent interest.

First, let us make the argument of A. Patrascioiu and E. Seiler in point 5) a bit more precise (see e.g. p. 13 in\cite{seiler2003case} or argument C4 of \cite{Patriascioiu}).

\begin{claim}[Argument of A. Patrascioiu and E. Seiler]
	Fix some inverse temperature $\beta$ above the $BKT$ transition on $\Z^2$ for the $XY$-model. Now consider some collection of random translation invariant ergodic subgraphs $(G_\delta)_{\delta > 0}$ of $\Z^2$ on the same probability space such that for $\delta < \delta'$ we have $G_\delta \supseteq G_\delta'$ and {$\bigcup_{\delta>0} G_\delta= \Z^2$} almost surely. 
Then there is some $\delta_0$ such that for all $\delta < \delta_0$, the $XY$ model on $G_\delta$ of inverse temperature $\beta$ has no mass gap. 
\end{claim}

In fact, if instead of the standard $XY$ model, one considers an $XY$ model where the angles are restricted  as in the case of the above argument by A. Patrascioiu and E. Seiler (see also \cite{aizenman1994slow}), this claim might well be true. However, it cannot hold in the case of the original $XY$ model as shown by the counterexample provided in Theorem \ref{th.example}. 

\begin{figure}[h!]
	\includegraphics[width=0.48\textwidth]{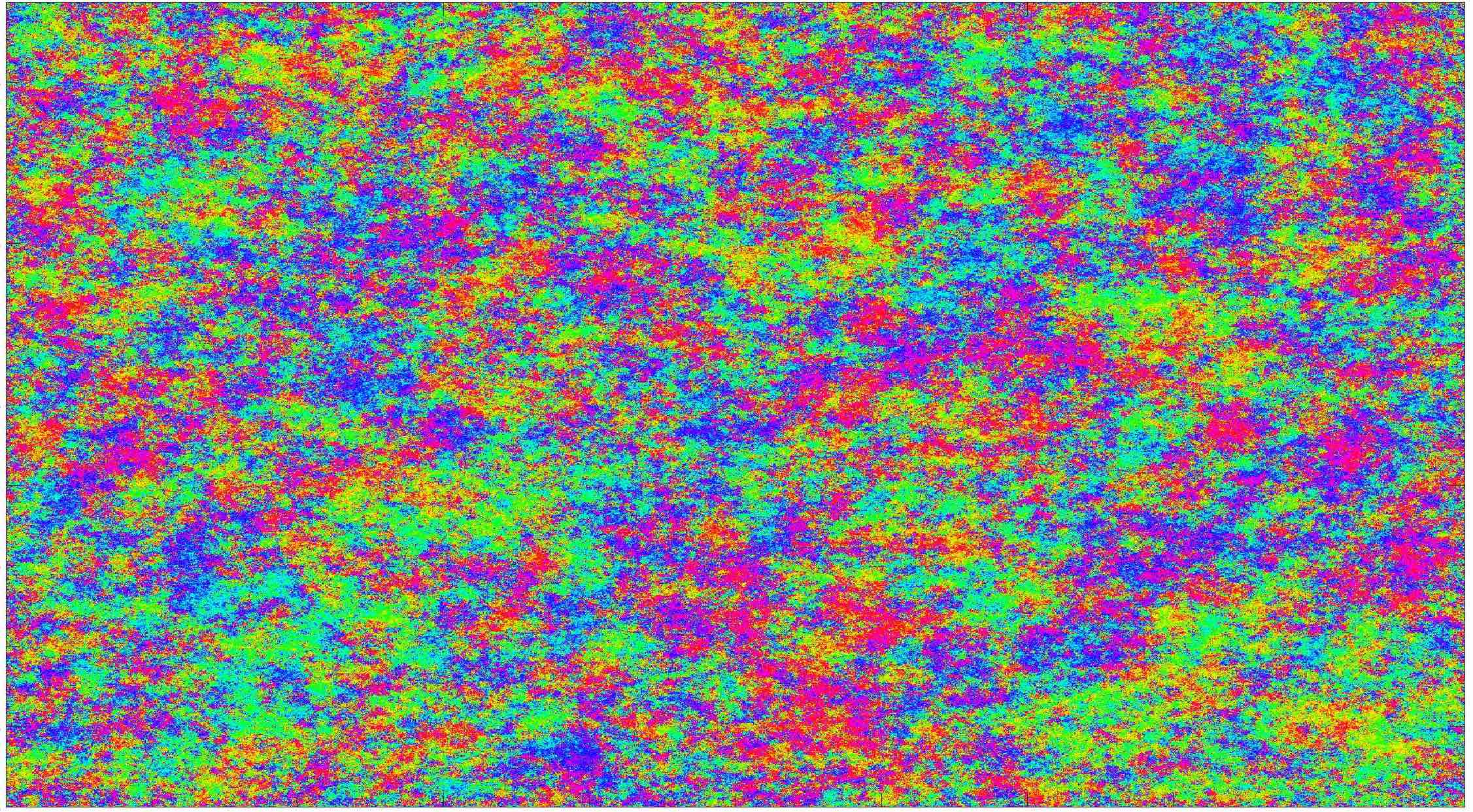}
		\includegraphics[width=0.48\textwidth]{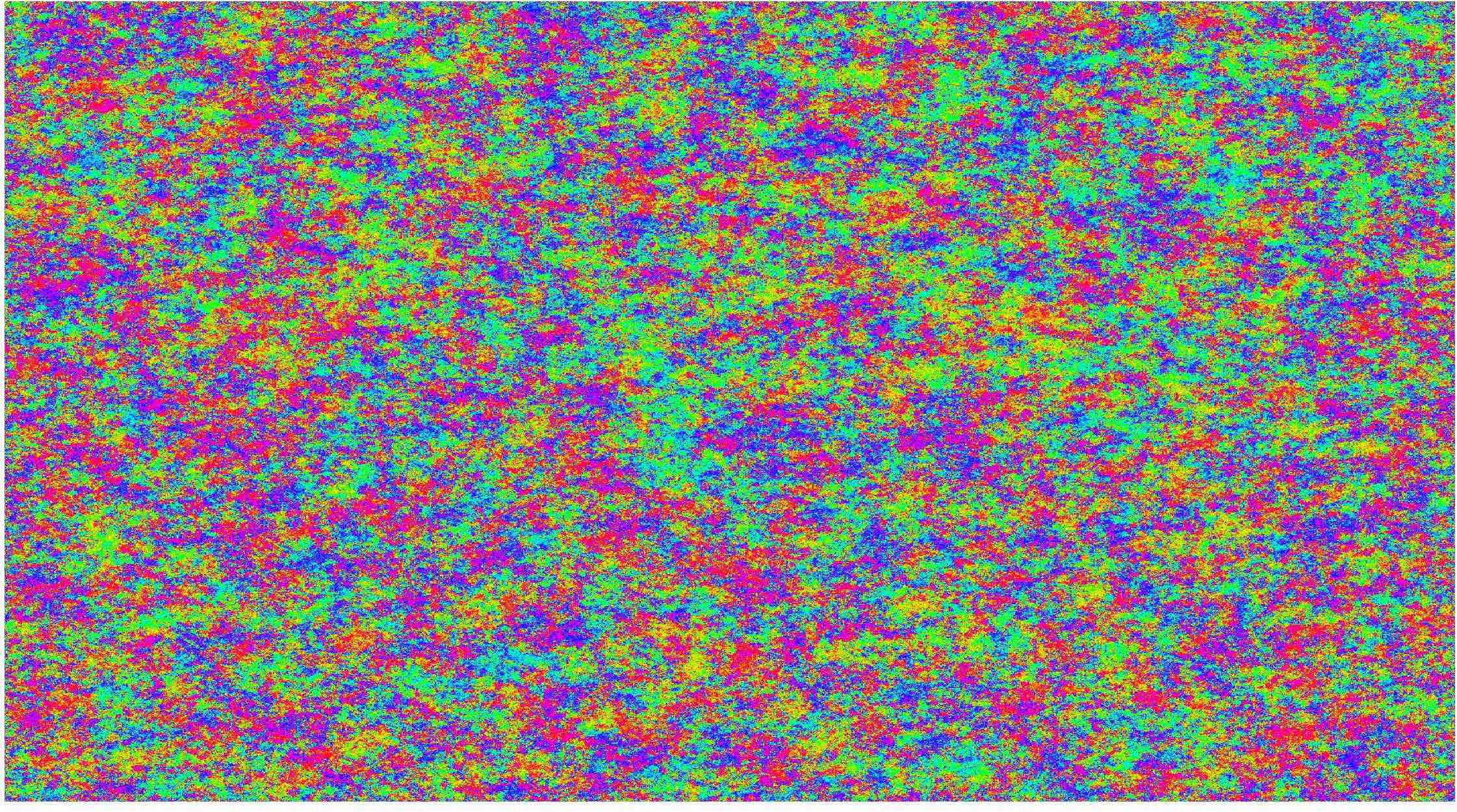}
			\includegraphics[width=0.48\textwidth]{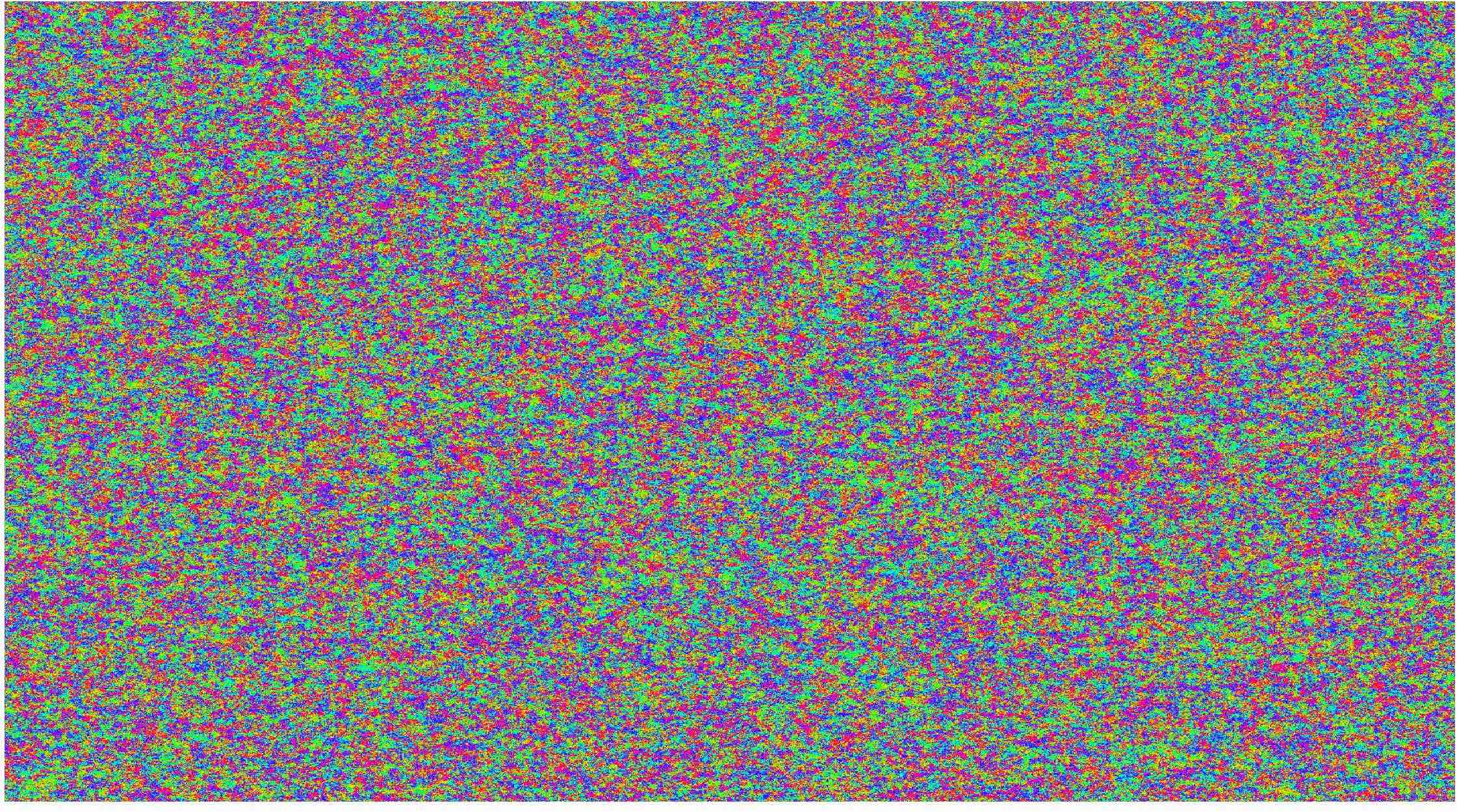}
			\caption{Angles of a massive GFF in $15000\times 15000$. From left to right, the mass is $5/15000$, $12.5/15000$ and $50/15000$.}
\end{figure}

\bigskip

\ni
{\em Proof of Theorem \ref{th.example}.}

\ni
\textbf{Step 1. Construction of the (coupled) graphs.}

Let us couple on the same probability space $m$-massive GFF $\phi^{(m)}: \Z^2 \to \R^N$ (recall definition \ref{d.GFF}) for all $m>0$ in such a way that  
for each $x\in \Z^2$ and each $1\leq i \leq N$, the process
\begin{align*}\label{}
m \mapsto [\phi^{(m)}_i(x)]^2
\end{align*}
is decreasing as a function of the mass $m$. This joint coupling follows from Lemma \ref{l.coupling}. 

For any fixed value of $\beta$, we define the sub-graph $G_m$ of $\Z^2$ to be 
\begin{align}\label{e.Gm}
G_m=G_m(\beta):= \{ x\in \Z^2, \| \phi^{(m)}(x) \|_2 > \sqrt{\beta} \}\,. 
\end{align}
From this definition, we readily obtain the following properties:
\bnum
\item The random graphs $G_m$ are translation-invariant and strongly mixing (see e.g. \cite{rosenblatt72} Thm 3.1 and discussion below that). 
\item By construction, it is clear also that $G_m \supset G_{m'}$ for any $m<m'$. 
\enum

\ni
\textbf{Step 2. Exponential decay for the $XY$ model on $G_m$.}

This property follows easily given our construction of $G_m$. Indeed, for any $x,y \in \Z^2$, we have thanks to Ginibre's inequality (\cite{Ginibre}) and Proposition \ref{pr.anglesGFF}:
\begin{align*}
\Eb{\<{\sigma_x \cdot \sigma_y}_{G_m, \beta} 1_{x,y \in G_m } }
& \leq \Eb{ \frac{\phi^{(m)}(x)} {\| \phi^{(m)}(x)\|} \cdot  \frac{\phi^{(m)}(y)} {\| \phi^{(m)}(y)\|}  1_{x,y \in G_m } } \\
& \leq \frac 1 {\beta} \Eb{ \phi^{(m)}(x) \cdot \phi^{(m)}(y)   1_{x,y \in G_m } }  \text{ (since $\|\phi^{(m)}\| > \sqrt{\beta}$ on $G_m$)} \\
& \leq \frac 1 {\beta} \Eb{ \phi^{(m)}(x) \cdot \phi^{(m)}(y)   } \\
& \leq \frac 1 {\beta} \exp(- C m \| x-y\|_2)\,.
\end{align*}
To conclude, we observe that 
\begin{align*}\label{}
\Eb{\<{\sigma_x \cdot \sigma_y}_{G_m, \beta} \md  x,y \in G_m  } &= \frac 1 {\Pb{x,y \in G_m}} \Eb{\<{\sigma_x \cdot \sigma_y}_{G_m, \beta} 1_{x,y \in G_m } }\,,
\end{align*}
and as we shall see below, $\Pb{x,y \in G_m} \to 1$ uniformly in $x,y$ as $m\to 0$. 

\bigskip
\ni
\textbf{Step 3. The graphs $G_m$ eventually cover the whole plane.}

This is again straight-forward from our construction. To show that a.s. $\bigcup_{m>0} G_m = \Z^2$, it is enough to show that for any fixed $x_0 \in \Z^2$, a.s. $x_0 \in \bigcup_{m > 0}G_m$. Now, $\{x_0 \notin G_m\} = \{\| \phi^{(m_k)} (x_0) \| \leq \sqrt{\beta}\}$. But the probability of this event is upper-bounded by 
\begin{align*}\label{}
O(1) \beta^{N/2} \log(1/m)^{-N/2}\,
\end{align*}
and hence goes to zero as $m \to 0$. As $G_m$ are increasing, we conclude that $\P(x_0 \in \bigcup_{m>0} G_m) = \lim_{m \to \infty} \P(x_0 \in G_m) = 1$.

\bigskip
Items $ii)$, $iii)$ and $iv)$ are the non-trivial part of the construction. Item $ii)$ has already been proved and was the main content of Theorem \ref{th.GFFmGFF}; 
$iv)$ will follow from a standard argument once we have established $iii)$.

\bigskip
\ni
\textbf{Step 4. Proof that $p_c(G_m) \to p_c(\Z^2)$.}

We can use  the same strategy as in Proposition \ref{pr.percolation_GFF} and Proposition \ref{pr.percolation_massive_GFF}. Indeed, we have that $p_c(G_m) > p_c(\Z^2)$. To prove that at any fixed $\beta$, we have $ p_c(G_m) \to p_c(\Z^2)$ as $m \to \infty$, we prove a dual statement. More precisely, it suffices to show that for every $p > p_c(\Z^2) = 1/2$, the random subsets $D_m$ given by 1-neighbourhoods of $G_m^c \cup \omega_{q}$, where $\omega_{q}$ is an independent Bernoulli percolation of parameter $q=1-p < 1/2$ is exponentially clustering. 

This follows exactly like in Proposition \ref{pr.percolation_GFF}, once we prove the following equivalent of Claim \ref{cl.openprob} (and of 	Claim \ref{cl.massive}):
\begin{claim}\label{cl.Perco}
For every $q < 1/2$ and every $\eps > 0$, there exists a scale $n=n(\beta,p,\eps)$ and a mass $\bar m=\bar m(\beta,p,\eps)$ such that for any $m\leq \bar m$, the following holds:
$$\P \otimes \P_q (\square_i \text{ is open}) = \P(\text{there is an open path in }D_m\text{ crossing }\An_i) < \eps.$$
Here the first probability measure samples the random graph $G_m$ (recall that its definition in ~\eqref{e.Gm} depends on the value of $\beta$ and the mass $m$) while the second one samples an independent percolation configuration $\omega_q$. 
\end{claim}

{\em Proof of Claim \ref{cl.Perco}.} 
First, using Theorem \ref{th.sharp}, as $q < 1/2$, we have that there exists $C=C(p)>0$ such that for all $n$ large enough.
\begin{align*}\label{}
\FK{q}{}{\text{ some cluster of }\omega_q\text{ in $\An_i$ has diameter larger than }C \log n} < \eps/2\,.
\end{align*}
But now we can just use Claim \ref{cl.massive} choosing $k(n) = C\log n$.
\qed

\bigskip
\ni
\textbf{Step 5. Long-range order for the Ising model $G_m$ at $\beta^{Ising}=1$.}

Finally, we show that one can choose $m$ sufficiently small so that the Ising model on $G_m$ will a.s. have long-range order already at $\beta^{Ising}=1$. (And therefore, $\beta_c^{Ising}(G_m) \leq 1$ a.s.) 
 
Let us consider the infinite volume FK-Ising percolation on $G_m$ with free boundary conditions at infinity (The FK measure on $G_m$ should be unique anyway, but we do not know it yet at this stage) and with parameter $p(\beta^{Ising}):=1-e^{-2 \beta^{Ising}}$. This measure is well defined by standard monotony arguments.
We will use the well-known fact that this measure  stochastically dominates a standard edge percolation on $G_m$ with intensity 
\begin{align}\label{e.pFK}
p= \frac{p(\beta^{Ising})} {p^{\beta^{Ising}} + (1-p(\beta^{Ising}) ) 2} =  \frac{1-e^{- 2 \beta^{Ising}}}
{1+e^{-2 \beta^{Ising}}} \,.
\end{align}
(This can be seen for example by sampling edges one at a time). In compact notations, this means 
\begin{align*}\label{}
\P_p \preccurlyeq \P_{p(\beta), q=2}^{\mathrm{FK}}\,.
\end{align*}
The important fact for us here is that this ratio in~\eqref{e.pFK} is stricly larger than $p_c(\Z^2)=1/2$ for any $\beta^{Ising}\geq 1$ (In fact for any $\beta^{Ising}\geq \frac {\log 3} 2 \approx 0.55$). 
So in particular at $\beta^{Ising}=1$, FK-Ising percolation with free boundary conditions at infinity on $G_m$ stochastically dominates a supercritical $p$-percolation on $G_m$ with $p>1/2$. 

By the previous step (step 4) (and the proof of Theorem \ref{th.GFFmGFF}), we therefore obtain that   this FK-percolation on $G_m$ stochastically dominates an ergodic process with a unique strongly percolating cluster (in the sense that its complement has exponentially decaying components).  This implies the existence of a unique infinite strongly percolating FK-Ising cluster which ends our proof. 
\qed

\subsection*{Acknowledgments.}
We wish to thank Sébastien Ott and Alain-Sol Sznitman for very interesting and useful discussions. The research of J.A is supported by the Eccellenza grant 194648 of the Swiss National Science Foundation; J.A is also part of NCCR SwissMAP. The research of C.G. is supported by the Institut Universitaire de France (IUF) and the French ANR grant ANR-21-CE40-0003. The research of A.S is supported by Grant ANID AFB170001 and FONDECYT iniciación de investigación N° 11200085.
\bibliographystyle{alpha}
\bibliography{biblio}

\end{document}